\documentclass{amsart}
\usepackage{amssymb,amsmath,amsthm}
 \usepackage{setspace}
\usepackage[all, knot]{xy}
\xyoption{arc} \xyoption{color}\xyoption{line}
\usepackage{mathrsfs}
\usepackage{fullpage}
\renewcommand{\leq}{\leqslant}
\renewcommand{\geq}{\geqslant}

\newcommand{\supp}{\operatorname{supp}}
\newcommand{\diag}{\operatorname{diag}}
\newtheorem{proposition}{Proposition}
\newtheorem{theorem}{Theorem}
\newtheorem{lemma}{Lemma}

\newtheorem{definition}{Definition}

\begin{document}

\title[PML on manifolds with quasicylindrical ends]{Analysis of Perfectly Matched Layer operators for acoustic scattering on manifolds with quasicylindrical ends}
\author{Victor Kalvin}
\address{Department of Mathematics and Statistics, Concordia University, 1455 de
Maisonneuve West, Montreal, H3G 1M8 Quebec, Canada}

\email{vkalvin@gmail.com}

\begin{abstract} In this paper we prove stability and exponential convergence of the Perfectly Matched Layer (PML)
method for acoustic scattering on manifolds with axial analytic
quasicylindrical ends. These manifolds model long-range geometric
perturbations (e.g. bending or stretching) of tubular waveguides
filled with homogeneous or inhomogeneous media.

We construct non-reflective infinite PMLs replacing the metric on a
part of the manifold by a non-degenerate complex symmetric tensor
field. We prove that the problem with PMLs of finite length is
uniquely solvable and solutions to this problem locally approximate
scattered solutions with an error that exponentially tends to zero
as the length of PMLs tends to infinity.
\end{abstract}

\maketitle

 \vspace{.2cm}

\vspace{.2cm}

\noindent {\bf Journal:} J. Math. Pures Appl.,
http://dx.doi.org/10.1016/j.matpur.2012.12.001

\section{Introduction}
The motivation of this work comes from the problem of numerical
modeling of acoustic scattering in tubular waveguides geometrically
perturbed up to infinity (e.g. bent or stretched) and filled with
homogeneous or inhomogeneous media. In order to obtain a good
approximation of scattered waves by numerical solutions of a problem
with finite computational domain, waveguides should be truncated
without creating excessive reflections from the artificial boundary
of truncation. The idea is to place in front of the boundary of
truncation a layer strongly absorbing the scattered waves. Due to
the strong attenuation, the homogeneous Dirichlet boundary condition
is a suitable boundary condition on the boundary of truncation. This
truncation scheme supplemented with very special construction of the
layer is widely known as the Perfectly Matched Layer (PML) method,
originally introduced in~\cite{ref1}. The method is in common use
for numerical analysis of a wide class of problems. For some of
them, stability and convergence of the method have been proved
mathematically;~e.g.~\cite{ref2,ref4,Kalvin
SIMA,KalvinSiNum,KimPasciak2,LASSAS}. In the present paper we
develop the PML method for Neumann Laplacians on manifolds with
axial analytic quasicylindrical ends and prove  stability and
exponential convergence of the method. Neumann Laplacians model the
scattering problem  described in the beginning of introduction, see,
e.g.,~\cite{MatPar}.

As is known, construction of PMLs is closely related to  complex
scaling. Complex scaling involves complex dilation of variables and
has a long tradition in mathematical physics and numerical
analysis~\cite{Cycon,Simon Reed iv}. In this paper we construct PMLs
in a different way. Instead of complex dilation of variables we
replace the metric on a part of the manifold by a non-degenerate
complex symmetric tensor field. This approach is close in spirit
to~\cite{LASSAS} and can be understood as a deformation of the
Remanian geometry by means of the complex scaling. As a result of
this deformation all formulas for the quadratic form and for
coordinate representations of the Neumann Laplacian turn into the
corresponding formulas for the quadratic form and the
non-selfadjoint operator describing infinite PMLs. This essentially
simplifies construction and tractability of the formulas. Let us
stress here that due to variation of the metric along
quasicylindrical ends not only the Laplace-Beltrami operator but
also the Neumann boundary condition should be changed in PMLs. This
new effect leads to significant difficulties in analysis of the PML
method.

Relying on ideas of the Aguilar-Balslev-Combes-Simon theory of
resonances~\cite{Cycon,Simon Reed iv} we establish a limiting
absorption principle. As is typically the case, scattered solutions
satisfying the limiting absorption principle locally coincide with
solutions to the problem with infinite PMLs. The latter solutions
are of some exponential decay at infinity. These results are mainly
based on localization of the essential spectrum of non-selfadjoint
operators corresponding to the problem with infinite PMLs. Thanks to
the exponential decay of solutions in PMLs we can establish
stability and exponential convergence of the PML method by using
compound expansions. This is a further development of our scheme for
analysis of stability and exponential convergence of the PML
method~\cite{KalvinSiNum,Kalvin SIMA}, see also~\cite{KalvinCCM}.
The added difficulties are due to the new effect mentioned above. To
overcome these difficulties we develop our approach to construction
of PMLs and use non-homogeneous boundary value problems in
localization of the essential spectrum and in compound expansions.

This paper is organized as follows. In Section~\ref{s2} we introduce
manifolds with axial analytic quasicylindrical ends and consider an
illustrative example. Section~\ref{sCS} is devoted to construction
of infinite PMLs. In Section~\ref{s5} we localize the essential
spectrum of the operator modeling infinite PMLs conjugated with
exponent. In Section~\ref{s6} we establish a limiting absorption
principle and show that outgoing and incoming solutions are of some
exponential decay in PMLs. Finally, in Section~\ref{s7} we study the
problem with finite PMLs and prove stability and exponential
convergence of the PML method.

\section{Manifolds with axial analytic quasicylindrical ends}\label{s2}
Let $\Omega$ be a  compact (not necessarily simply connected)
$n$-dimensional manifold with smooth boundary $\partial\Omega$.
Denote by $\Pi$ the semi-cylinder $\Bbb R_+\times \Omega$, where
$\Bbb R_+$ is the positive semi-axis, and $\times$ stands for the
Cartesian product. Consider an oriented connected $n+1$-dimensional
manifold $\mathcal M$ representable in the form $\mathcal M=\mathcal
M_c\cup \Pi$, where $\mathcal M_c$ is a  compact manifold with
smooth boundary, see
 Fig.~\ref{manif}. We also assume that the boundary $\partial\mathcal
 M$ of $\mathcal M$ is smooth.
\begin{figure}
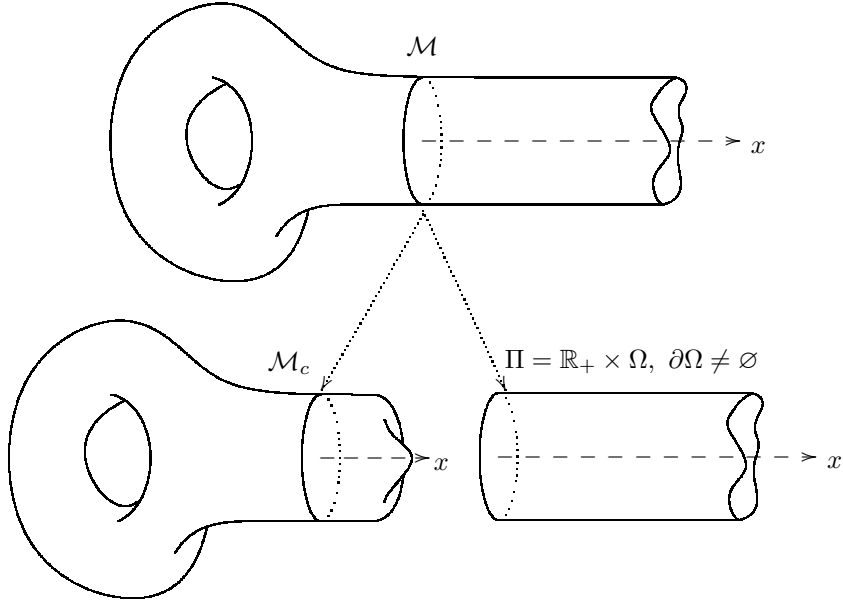
\centering
\[\xy0;/r.20pc/: (33,50)*{\xy
(40,10);(-18,-11)**\crv{(0,10)&(-13,10)&(-20,15)&(-30,25)&(-50,15)&(-50,-15)&(-30,-25)&(-20,-20)};
(38,-10);(-23,-15)**\crv{(0,-10)&(-10,-10)&(-20,-10)}; (-32,-10);
(-33,10)**\crv{(-25,-7)&(-25,8)}; (-29,-7);
(-31,9)**\crv{(-32,-9)&(-39,-2)&(-37,6)}; (40,10);
(38,-10)**\crv{(43,9)&(39,5)&(42,2)&(39,-2)&(42,-9)}; (40,10);
(38,-10)**\crv{(38,9)&(37,7)&(35,5)&(41,-2)&(35,-5)&(37,-10)};
{\ar@{-->} (0,0)*{}; (50,0)*{}}; (53,-1)*{x};
(0,0)*\ellipse(3,10){.}; (0,0)*\ellipse(3,10)__,=:a(-180){-};
(0,15)*{\mathcal M};
\endxy};
(0,0)*{\xy (8,10);
(-18,-11)**\crv{(0,10)&(-13,10)&(-20,15)&(-30,25)&(-50,15)&(-50,-15)&(-30,-25)&(-20,-20)};
(8,-10); (-23,-15)**\crv{(0,-10)&(-10,-10)&(-20,-10)}; (-32,-10);
(-33,10)**\crv{(-25,-7)&(-25,8)}; (-29,-7);
(-31,9)**\crv{(-32,-9)&(-39,-2)&(-37,6)}; (-5,15)*{\mathcal M_c};
(10,-7); (10,6)**\crv{(10,-6)&(14,-2)&(15,1)&(10,4)&(10,7)}; (8,10);
(13,2)**\crv{(12,10)}; (8,-10); (13,-3)**\crv{(12,-10)}; {\ar@{-->}
(0,0)*{}; (17,0)*{}}; (19,-1)*{x}; (0,0)*\ellipse(3,10){.};
(0,0)*\ellipse(3,10)__,=:a(-180){-};
\endxy};
(70,4)*{\xy (0,10); (40,10)**\dir{-};(0,-10); (38,-10)**\dir{-};
(0,0)*\ellipse(3,10){.}; (0,0)*\ellipse(3,10)__,=:a(-180){-};
(21,15)*{\Pi=\Bbb R_+\times\Omega,\
\partial\Omega\neq\varnothing};(0,15)*{\vphantom{\mathcal M_c}};
(40,10); (38,-10)**\crv{(43,9)&(39,5)&(42,2)&(39,-2)&(42,-9)};
(40,10);
(38,-10)**\crv{(38,9)&(37,7)&(35,5)&(41,-2)&(35,-5)&(37,-10)};
{\ar@{-->} (0,0)*{}; (50,0)*{}}; (53,-1)*{x};
(0,0)*\ellipse(3,10){.}; (0,0)*\ellipse(3,10)__,=:a(-180){-};
\endxy};
{\ar@{.>} (31,39)*{};(15,11)*{}};{\ar@{.>} (31,39)*{};(44,11)*{}};
\endxy\]
\caption{Representation $\mathcal M=\mathcal
M_{c}\cup\Pi$.}\label{manif}
\end{figure}

Let $\mathsf g$ 
be a Riemannian metric on $\mathcal M$. We identify the cotangent
bundle $\mathrm T^*\Pi$ with the tensor product $\mathrm T^*\Bbb
R_+\otimes\mathrm T^*\Omega$   via the natural isomorphism induced
by the product structure on $\Pi$. Then
\begin{equation}\label{split}
\mathsf g\!\upharpoonright_{\Pi}=\mathfrak g_0 dx\otimes dx+2
\mathfrak g_1\otimes dx +\mathfrak g_2, \quad\mathfrak g_k(x)\in
C^\infty\mathrm T^*\Omega^{\otimes k},\ x\in\Bbb R_+.
\end{equation}

Denote by  $\Bbb C\mathrm T^*\Omega^{\otimes k}$  the tensor power
of the complexified cotangent bundle $\Bbb C\mathrm T^*\Omega$ with
fibers $\Bbb C\mathrm T_{y}^*\Omega=\mathrm T_{
y}^*\Omega\otimes\Bbb C$. In what follows $C^m$ stands for sections
of complexified bundles. We equip the space $C^1\mathrm
T^*\Omega^{\otimes k}$ with the norm
\begin{equation}\label{Enorm}
\|\cdot\|_{\mathfrak e}=\max_{ y\in\Omega}\bigl(|\cdot|_{\mathfrak
e}(y)+|\nabla\cdot|_{\mathfrak e}(y)\bigr),
\end{equation}
where $\mathfrak e$ is a Riemannian metric on $\Omega$,
$|\cdot|_{\mathfrak e}(y)$ is the norm in  $\Bbb C\mathrm T_{
y}^*\Omega^{\otimes k}$, and $\nabla: C^1\mathrm T^*\Omega^{\otimes
k}\to C^0\mathrm T^*\Omega^{\otimes k+1}$ is the Levi-Civita
connection on $(\Omega,\mathfrak e)$. The norms induced by different
metrics $\mathfrak e$  are equivalent.

\begin{definition}\label{ACE}
We say that $(\mathcal M,\mathsf g)$ is a manifold with axial
analytic quasicylindrical end $(\Pi,\mathsf
g\!\upharpoonright_{\Pi})$, if the following conditions hold:
\begin{itemize}
\item[i.]  The coefficients $\Bbb R_+\ni x\mapsto \mathfrak g_k(x)\in C^\infty\mathrm T^*\Omega^{\otimes
k}$ in~\eqref{split} extend by analyticity from the semi-axis $\Bbb
R_+$ to the sector $\Bbb S_\alpha=\{z\in\Bbb C:|\arg
z|<\alpha<\pi/4\}$.
\item[ii.] The values $\|\mathfrak
g_0(z)-1\|_{\mathfrak e}$ , $\|\mathfrak g_1(z)\|_{\mathfrak e}$,
and $\|\mathfrak g_2(z)-\mathfrak h\|_{\mathfrak e}$  converge to
zero uniformly in $z\in\Bbb S_\alpha$ as $|z|\to\infty$, where
$\mathfrak h$ is a  metric on $\Omega$.


\end{itemize}

\end{definition}
Long-range geometric perturbations of tubular waveguides (i.e. of
manifolds $(\mathcal M,\mathsf g)$ with $\mathsf
g\!\upharpoonright_\Pi=dx\otimes dx+\mathfrak h$) are included into
consideration as there are no assumptions on the rate of convergence
  in condition ii of Definition~\ref{ACE}. Let us
give some illustrative examples of  manifolds  with axial analytic
quasicylindrical ends.

Let $ \Omega$ be a bounded domain in $\Bbb R^n$ with smooth boundary
($\Omega$ is a line segment if $n=1$). By $(x,y)$ and $(s,t)$ we
denote the Cartesian coordinates in $\Bbb R^{n+1}$, where
$x,s\in\Bbb R$ and $y,t\in\Bbb R^n$. Consider a closed domain
$\mathcal M$ with  smooth boundary such that $\{(x,y)\in {\mathcal
M}: x\leq 0\}$ is a bounded subset of the half-space $\{(x,y)\in
\Bbb R^{n+1}: x>-2\}$ and  $\{(x,y)\in {\mathcal M}:x> 0\}$ is the
semi-cylinder $\Pi$. Let
$$
\mathcal G=\{(s,t)\in\Bbb R^{n+1}: (s,t)=\phi(x,y), (x,y)\in\mathcal
M\}
$$
be the image of $\mathcal M$ under a diffeomorphism $\phi$
satisfying the following conditions:
\begin{enumerate}
\item[\it i.] The
function $x\mapsto \phi(x,\cdot)\in C^\infty(\Omega)$ has an
analytic continuation from $\mathbb R_+$ to the sector $\Bbb
S_\alpha$;
 \item[\it ii.] The element in the row $\ell$ and column $m$ of the matrix $(\phi'(z,\cdot))^t\phi'(z,\cdot)$,
 where $(\phi')^t$ is the transpose of the  Jacobian  $\phi'$,  uniformly tends to
 the Kronecker delta $\delta_{\ell m}$ in the space $C^1(\Omega)$ as $|z|$ tends to
infinity, $z\in\Bbb S_\alpha$.
\end{enumerate}
 For instance, we can take
 \begin{equation*}
\phi(x,y)=\left( x,(x+3)^\beta a+ (1+(x+3)^{\gamma} )y\right),\quad
a\in\Bbb R^n,\ \beta<1,\ \gamma<0.
\end{equation*}
Then the boundary $\partial\mathcal G$ approaches at infinity the
bent semi-cylinder $\{(s,t): s\in\Bbb R_+, t-(s+3)^\beta
a\in\Omega\}$. If we take
\begin{equation*}
\phi(x,y)=\left(\int_0^x \bigl(1+1/\log(\tilde x+4)\bigr)\, d\tilde
x, \bigl(1+1/\log(x+5)\bigr)y\right),
\end{equation*}
then $\partial\mathcal G$ slowly approaches at infinity the
semi-cylinder $\Bbb R_+\times\partial\Omega$, cf. Fig.~\ref{domain}.
\begin{figure}[h]
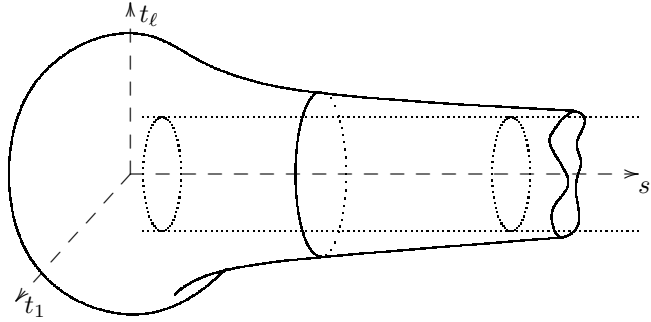

\[\xy0;/r.20pc/:  (40,10);(-15,-15)**\crv{(0,12)&(-13,15)&(-20,18)&(-30,25)&(-50,15)&(-50,-15)&(-30,-25)&(-20,-20)};
(38,-10);(-23,-19)**\crv{(0,-13)&(-10,-14)&(-20,-16)};
(40,10); (38,-10)**\crv{(43,9)&(39,5)&(42,2)&(39,-2)&(42,-9)};
(40,10);
(38,-10)**\crv{(38,9)&(37,7)&(35,5)&(41,-2)&(35,-5)&(37,-10)};
{\ar@{-->} (-30,0)*{}; (50,0)*{}}; (51,-2)*{s};
(0,0)*\ellipse(4,13){.}; (0,0)*\ellipse(4,13)__,=:a(-180){-};
{\ar@{-->} (-30,0)*{}; (-30,27)*{}}; (-27,25)*{t_\ell}; {\ar@{-->}
(-30,0)*{}; (-48,-20)*{}}; (-45,-21)*{t_1};
{\ar@{.}(-28,9); (50,9)}; {\ar@{.}(-28,-9); (50,-9)};
(-25,0)*\xycircle(3,9){.}; (30,0)*\xycircle(3,9){.};
\endxy\]
\caption{Domain $\mathcal G\subset \Bbb R^{n+1}$ with
quasicylindrical end.}\label{domain}
\end{figure}
Time-harmonic acoustic scattering at angular frequency $\omega$ in
the domain $\mathcal G$ filled by an inhomogeneous medium with
anisotropic density tensor $ \rho$ and bulk modulus $K$ is modeled
by the Neumann Laplacian on $(\mathcal G, K\rho)$ with spectral
parameter $\mu=\omega^2$, where $K\rho$ is a metric on $\mathcal G$,
e.g.,~\cite{MatPar}. Under certain assumptions on regularity of
$\rho$ and $K$ the pullback $\mathsf g$ of  $K\rho$ by the
diffeomorphism
 $\phi$ satisfies the conditions of
Definition~\ref{ACE}. These assumptions are apriori met for
homogeneous isotropic media, i.e. the pullback $\mathsf g$ of the
Euclidean metric $K\rho$ by the diffeomorphism $\phi$ satisfies the
conditions of Definition~\ref{ACE}. To study the Neumann Laplacian
on $(\mathcal G,K\rho)$ is the same as to study the Neumann
Laplacian on $(\mathcal M,\mathsf g)$. For our aims it is natural
and convenient to work on manifolds with axial analytic
quasicylindrical ends.

\section{Construction of infinite PMLs}\label{sCS}
Our approach here is very close in spirit to~\cite{LASSAS}. Let
$\mathsf s_{r}(x)=\mathsf s(x-{r})$, where ${r}>0$ is sufficiently
large and $\mathsf s\in C^\infty(\Bbb R)$ possesses the properties:
\begin{equation}\label{ab}
\begin{aligned}
&\mathsf s(x)=0 \text{ for all } x\leq 1,\\
&0\leq \mathsf s'(x)\leq 1 \text{ for all } x\in\Bbb R, \text{ and }
\mathsf s'(x)= 1 \text{ for large
 } x>0;
\end{aligned}
\end{equation}
here $\mathsf s'=\partial \mathsf s/\partial x$. Let  $\mathrm
T'^*\Bbb S_\alpha$ be the holomorphic cotangent bundle $\{(z,
c\,dz):z\in\Bbb S_\alpha, c\in\Bbb C\}$, where $dz=d\Re z+id\Im z$.
Consider the tensor field
\begin{equation}\label{TF}
 \mathfrak g_0 dz\otimes dz +2\mathfrak g_1\otimes dz+\mathfrak g_2\in C^\infty({\mathrm T}'^*\Bbb
S_\alpha\otimes{\mathrm T}^*\Omega )^{\otimes 2}
\end{equation}
with analytic coefficients $\Bbb S_\alpha\ni z\mapsto \mathfrak
g_k(z)\in C^\infty \mathrm T^*\Omega^{\otimes k}$, cf.
Definition~\ref{ACE}. For all values of $\lambda$ in the disk
\begin{equation}\label{disk}
\mathcal O_\alpha=\{\lambda\in\mathbb C: |\lambda|<\sin\alpha<
1/\sqrt{2}\}
\end{equation}
the function $\mathsf s_{r}$  defines the embedding
\begin{equation}\label{emb}
\mathrm T^*\Bbb R_+\ni (x, a\,dx)\mapsto \bigl(x+\lambda \mathsf
s_{r}(x), a(1+\lambda \mathsf s'_{r}(x))^{-1} d z\bigr)\in \mathrm
T'^*\Bbb S_\alpha,
\end{equation}
where $|1+\lambda \mathsf s'_{r}(x)|>1-1/\sqrt{2}$.  This embedding
together with~\eqref{TF} induces the tensor field
\begin{equation}\label{TFL}
\mathsf g_\lambda\!\upharpoonright_\Pi=\mathfrak g^{r}_{0,\lambda}
dx\otimes dx +2\mathfrak g^{r}_{1,\lambda} \otimes dx +\mathfrak
g^{r}_{2,\lambda}\in C^\infty {\mathrm T}^* \Pi^{\otimes 2},
\end{equation}
where $\mathfrak g^{r}_{k,\lambda}(x)\in C^\infty \mathrm
T^*\Omega^{\otimes k}$ are  smooth in $x\in\Bbb R_+$ and analytic in
$\lambda\in\mathcal O_\alpha$ coefficients given  by
$$
 \mathfrak
g^{r}_{k,\lambda}(x)=(1+\lambda \mathsf s'_{r}(x))^{2-k}\mathfrak
g_{k}(x+\lambda \mathsf s_{r}(x)).
$$
Note that $\mathsf g_\lambda\!\upharpoonright_\Pi$ with
$\lambda\in\mathcal O_\alpha\cap\Bbb R$ is the pullback of the
metric $\mathsf g\!\upharpoonright_\Pi$ by  the diffeomorphism
$\varkappa_\lambda (x,y)=(x+\lambda \mathsf s_{r}(x),y)$ scaling the
semi-cylinder $\Pi$ along its axis $\Bbb R_+$. Since $ \mathsf
s_{r}$ is supported on $({r},\infty)$, the equality $\mathsf
g_\lambda\!\upharpoonright_{(0,{r})\times\Omega}=\mathsf
g\!\upharpoonright_{(0,{r})\times\Omega}$ holds for all
$\lambda\in\mathcal O_\alpha$. We extend $\mathsf
g_\lambda\!\upharpoonright_{\Pi}$  to $\mathcal M$ by setting
$\mathsf g_\lambda\!\upharpoonright_{\mathcal M\setminus\Pi}=\mathsf
g\!\upharpoonright_{\mathcal M\setminus\Pi}$. As a result we obtain
the analytic function
$$
\mathcal O_\alpha\ni\lambda\mapsto \mathsf g_\lambda \in C^\infty
\mathrm T^*\mathcal M^{\otimes 2}.
$$
Clearly, $\mathsf g_0=\mathsf g$ and~\eqref{TFL} with $\lambda=0$ is
the same as~\eqref{split}. By analyticity in $\lambda$ we conclude
that $\mathsf g_\lambda$ is a symmetric tensor field. The Schwarz
reflection principle gives $\overline{\mathsf g_\lambda}=\mathsf
g_{\overline{\lambda}}$. It must be stressed that $\mathsf
g_\lambda$ with $\lambda\neq 0$ depends on  ${r}$, however we do not
indicate this for brevity of notations.  As ${r}$ is sufficiently
large, $\mathsf g_\lambda$ is non-degenerate; i.e. for
 $p\in\mathcal M$ and any nonzero $a\in \Bbb
C\mathrm T_p\mathcal M$ there exists $b\in \Bbb C\mathrm T_p\mathcal
M$ such that $\mathsf g^p_\lambda[a,b]\neq 0$, where $\mathsf
g_\lambda^p[\cdot,\cdot]$ is the  sesquilinear form naturally
defined by $\mathsf g_\lambda$. Indeed, on $\mathcal
M\setminus({r},\infty)\times\Omega$ the tensor field $\mathsf
g_\lambda$ is non-degenerate because it coincides there with the
metric $\mathsf g$, while on  $({r},\infty)\times\Omega$ it is only
little different from the non-degenerate tensor field
$(1+\lambda\mathsf s_{r}')^2dx\otimes dx+\mathfrak h$,
see~\eqref{TFL} and condition ii of Definition~\ref{ACE}. In
particular, $\mathsf g_\lambda\upharpoonright_\Pi$ coincides with
$(1+\lambda\mathsf s_{r}')^2dx\otimes dx+\mathfrak h$ in the case of
manifold with cylindrical end, i.e. in the case  $\mathsf
g\upharpoonright_\Pi=dx\otimes dx+\mathfrak h$.

It is well known that a metric induces the musical isomorphism
between tangent and cotangent bundles. Similarly,
 $\mathsf g_\lambda$ induces the fiber isomorphism
$$
\Bbb C\mathrm T^*_p\mathcal M\ni\xi \mapsto {^\lambda\sharp} \xi\in
\Bbb C \mathrm T_p\mathcal  M,
$$
where ${^\lambda\sharp} \xi$ is a unique vector satisfying the
equality $ \xi\overline{a}=\mathsf g_\lambda^p [{^\lambda\sharp}
\xi,a]$ for all $a \in \Bbb C \mathrm T_p\mathcal  M$. We extend
$\mathsf g_\lambda^p[\cdot,\cdot]$ to the pairs $(\xi,\eta)\in {\Bbb
C{\mathrm T}^*_{p}\mathcal M}\times {\Bbb C{\mathrm T}^*_{p}\mathcal
M}$ by setting
\begin{equation*}
{\mathsf g_\lambda^p} [\xi,\eta]={\mathsf g_\lambda^p}
[{^\lambda\sharp}\,\xi,{^{\bar{\lambda}}\sharp}\,\eta], \quad
\lambda\in\mathcal O_\alpha.
\end{equation*}
 If $\lambda$ is real, then ${\mathsf g_\lambda^p} [\cdot,\cdot]$ is the
positive Hermitian form induced by the  metric $\mathsf g_\lambda$.
The corresponding volume form $\operatorname{dvol}_{\lambda}$
extends by analyticity to all $\lambda\in\mathcal O_\alpha$.
Introduce the  deformed global inner product
\begin{equation*}
(\xi,\omega)_\lambda=\int_{\mathcal M} \mathsf g_\lambda[\xi,\omega]
\,\operatorname{dvol}_{\lambda},\quad \lambda\in\mathcal O_\alpha,
\quad \xi,\omega\in C_c^\infty\mathrm T^*\mathcal M^{\otimes
k},\quad k=0,1.
\end{equation*}
Let us stress that  for non-real $\lambda\in\mathcal O_\alpha$ the
form $\operatorname{dvol}_{\lambda}$ is complex-valued and the
deformed inner product
$(\xi,\omega)_\lambda=\overline{(\omega,\xi)_{\bar{\lambda}}}$ is
not Hermitian. Let us consider the sesquilinear quadratic form
\begin{equation*}
\mathsf
q_\lambda[u,u]=\bigl(du,d(\varrho_{\bar{\lambda}}\,u)\bigr)_\lambda,\quad
u\in C_c^\infty(\mathcal M),\quad\lambda\in\mathcal O_\alpha,
\end{equation*}
where $d: C_c^\infty(\mathcal M)\to C_c^\infty \mathrm T^*\mathcal
M$ is the exterior derivative and $\varrho_\lambda\in
C^\infty(\mathcal M)$ is such that
$\varrho_\lambda\operatorname{dvol}_{\lambda}=\operatorname{dvol}_{0}$.
\begin{lemma}\label{L1} Introduce the space $L^2(\mathcal M)$
and the Sobolev space $H^1(\mathcal M)$
 as completion of the set $C_c^\infty(\mathcal M)$ with respect to the
 norms
$$
\|u\|=\sqrt{(u,u)_0}, \quad \|u\|_{ H^1(\mathcal
M)}=\sqrt{(du,du)_0+(u,u)_0}
$$
respectively. Then the family $\mathcal O_\alpha\ni\lambda\mapsto
\mathsf q_\lambda$ of unbounded quadratic forms in $L^2(\mathcal M)$
with domain $H^1(\mathcal M)$ is analytic in the sense of
Kato~\cite{Kato,Simon Reed iv}; i.e. $\mathsf q_\lambda$ is a closed
densely defined sectorial form, and the function $\mathcal
O_\alpha\ni\lambda\mapsto\mathsf q_\lambda[u,u]$ is analytic for any
$u\in H^1(\mathcal M)$. Moreover, the sector of $\mathsf q_\lambda$
is independent of $\lambda\in\mathcal O_\alpha$.
\end{lemma}

The  family $\mathcal O_\alpha\ni\lambda\mapsto \mathsf q_\lambda$
uniquely determines an analytic family $\mathcal
O_\alpha\ni\lambda\mapsto {^\lambda\!\Delta}$ of unbounded
m-sectorial operators in $L^2(\mathcal M)$~\cite{Kato,Simon Reed
iv}. (Here and elsewhere m-sectorial means that the numerical range
$\{(Au,u): u\in \mathscr D(A)\}$ and the spectrum $\sigma(A)$ of a
closed unbounded operator $A$ with the domain $\mathscr D(A)$ are
both in some sector $\{\mu\in\Bbb C:|\arg(\mu+c)|\leq
\vartheta<\pi/2, c>0\}$.)  We have
\begin{equation}\label{LF}
({^\lambda\!\Delta}u,v)_0=\mathsf q_\lambda[u,v],\quad u\in\mathscr
D({^\lambda\!\Delta}), v\in H^1(\mathcal M).
\end{equation}
 In particular,
$({^0\!\Delta}u,v)_0=( du,dv)_0$ and ${^0\!\Delta}$ is the
selfadjoint Neumann Laplacian. We consider ${^\lambda\!\Delta}$ with
non-real $\lambda\in\mathcal O_\alpha$ as the operator modeling an
infinite PML on $({r},\infty)\times\Omega$. We will show that this
PML is an artificial nonreflective strongly absorbing layer for the
outgoing (resp. incoming) solutions if $\Im\lambda>0$ (resp.
$\Im\lambda<0$). Note that for $u\in\mathscr D({^0\!\Delta})$
supported outside $({r},\infty)\times\Omega$ we have $u\in \mathscr
D({^\lambda\!\Delta})$ and ${^0\!\Delta}u\equiv
{^\lambda\!\Delta}u$.

\begin{proof} In a finite covering of $\Omega$
by coordinate neighborhoods $\{\mathscr U_j\}$ take a neighborhood
$\mathscr U_j$ and coordinates $y\in\Bbb R^n$ in $\mathscr U_j$.
Then $\{\Bbb R_+\times\mathscr U_j\}$ is a covering of $\Pi$ and
$(x,y)$ are coordinates in $\Bbb R_+\times \mathscr U_j$. Denote
$\partial_0=\partial/\partial x$, $\partial_\ell=\partial/\partial
y_\ell$, $d_0=dx$, and $d_\ell=dy_\ell$. As $\mathsf g_\lambda$ is
non-degenerate, the matrix $\mathsf g_{\lambda,\ell m}$ in the
coordinate representation $\mathsf g_\lambda=\mathsf g_{\lambda,\ell
m}d_\ell\otimes d_m$ has the inverse $\mathsf g_{\lambda}^{\ell m}$
and for $\xi=\xi_\ell d_\ell$ we have ${^\lambda\sharp}\xi= \mathsf
g_{\lambda}^{\ell m} \xi_\ell\partial_m$. Hence $\mathsf
g_\lambda[\xi,\xi]=\mathsf g_{\lambda}^{\ell m} \xi_\ell \bar\xi_m$.

 On $\Omega$ all metrics are equivalent. In particular we can take $\mathfrak e=\delta_{\ell
m}d_\ell\otimes d_m$, where $\delta_{\ell m}$ is the Kronecker
delta. Then Definition~\ref{ACE} together with~\eqref{TFL}
immediately implies that
\begin{equation}\label{STAB-}
\bigr|\partial_k^j\bigr(\mathsf g_{\lambda}^{\ell
m}-\diag\{(1+\lambda)^{-2},\mathfrak h^{-1}\}_{\ell
m}\bigr)(x,y)\bigr|\leq C(x)\to 0\text{ as } x \to\infty,\ j=0,1.
\end{equation}
Moreover,
\begin{equation}\label{STAB}
\bigr|\partial_k^j\bigr(\mathsf g_{\lambda}^{\ell
m}-\diag\{(1+\lambda s'_{r})^{-2},\mathfrak h^{-1}\}_{\ell
m}\bigr)(x,y)\bigr|\leq \epsilon_{r},\quad x\geq {r},\
\lambda\in\mathcal O_\alpha,\ j=0,1,
\end{equation}
with sufficiently small $\epsilon_{r}$ as ${r}$ is sufficiently
large. Since $\mathfrak h^{-1}(y)$ is a symmetric positive definite
matrix, we get
$$
\bigl|\arg\bigl((1+\lambda s'_{r}(x))^{-2}|\xi_0|^2+\mathfrak
h^{\ell m}(y)\xi_\ell\bar\xi_m\bigr)\bigr|<2\alpha<\pi/2,
$$
$$
c|\xi|^2\leq \bigl|(1+\lambda s'_{r}(x))^{-2}|\xi_0|^2+\mathfrak
h^{\ell m}(y)\xi_\ell\bar\xi_m\bigr|\leq |\xi|^2/c,
$$
where $|\xi|^2=\xi_\ell\bar\xi_\ell$ and $c>0$. This together
with~\eqref{STAB} gives
\begin{equation}\label{++}
\delta|\xi|^2\leq \Re\bigl(\mathsf g_\lambda^{\ell
m}(x,y)\xi_\ell\bar\xi_m\bigr)\leq \delta|\xi|^2,\quad |\arg
\bigl(\mathsf g_\lambda^{\ell m}(x,y)\xi_\ell\bar\xi_m\bigr)|\leq
\vartheta,
\end{equation}
where $x\geq {r}$,
$\delta=\min\{c-\epsilon_{r},(1/c+\epsilon_{r})^{-1}\}$, and
$\vartheta=2\alpha+2\arcsin(\epsilon_{r}/2c)<\pi/2$.

 From~\eqref{++}
and $\mathsf g^p_\lambda=\mathsf g^p$, $p\in\mathcal
M\setminus({r},\infty)\times\Omega$, we conclude that the first term
in the representation
\begin{equation}\label{+++}
\mathsf q_\lambda[u,u]=\int_{\mathcal M} \mathsf
g_\lambda[du,du]\operatorname{dvol}_0+\int_{\mathcal M} \frac
{\mathsf g_\lambda[du,ud\varrho_{\bar \lambda}]}
{\varrho_\lambda}\operatorname{dvol}_0,\quad u\in
C_c^\infty(\mathcal M),
\end{equation}
meets the estimates
\begin{equation}\label{+++1}
-\theta\leq\arg \int_{\mathcal M} \mathsf
g_\lambda[du,du]\operatorname{dvol}_0\leq \theta,\quad
c(du,du)_0\leq \Re \int_{\mathcal M} \mathsf
g_\lambda[du,du]\operatorname{dvol}_0\leq (du,du)_0/c
\end{equation}
with some $\theta<\pi/2$ and $c>0$. The second term has an arbitrary
small relative bound with respect to the first term. Indeed,
\begin{equation}\label{+++2}
\begin{aligned}
&\left|\int_{\mathcal M} \frac {\mathsf g_\lambda[du,ud\varrho_{\bar
\lambda}]} {\varrho_\lambda}\operatorname{dvol}_0\right|\leq\frac
{(1+\tan\theta)} {\inf_{p\in\mathcal
M}\varrho_\lambda(p)}\left(\Re\int_{\mathcal M}\mathsf
g_\lambda[du,du]\operatorname{dvol}_0\right)^{1/2}\\
&\quad \times\left(\Re\int_{\mathcal M} |u|^2 \mathsf g_\lambda[d
\varrho_{ \lambda},d\varrho_{
\lambda}]\operatorname{dvol}_0\right)^{1/2} \leq \epsilon
\left|\int_{\mathcal M}\mathsf g_\lambda[du,du]\operatorname{dvol}_0
\right|+\epsilon^{-1}C\|u\|^2
\end{aligned}
\end{equation}
for any $\epsilon>0$ and an independent of $\epsilon$, $\lambda$,
and $u$ constant $C$. Here we used the uniform in $\lambda$ and $p$
estimates $|\varrho_\lambda(p)|\geq c>0$ and $\mathsf g^p_\lambda[d
\varrho_{ \lambda},d\varrho_{ \lambda}]\leq C<\infty$, which are
valid because $\varrho_\lambda(p)=1$ for $p\notin
({r},\infty)\times\Omega$ and $\varrho_\lambda=\sqrt{|\mathsf
g_0|/|\mathsf g_\lambda|}$ in $\Bbb R_+\times\mathscr U_j$, where
$\partial_k^j|\mathsf
g_\lambda|(x,y)\to\partial_k^j\bigl((1+\lambda\mathsf s'_{r})^2
|\mathfrak h|\bigr)(x,y)$ as $x\geq r\to\infty$, cf.~\eqref{STAB}
(as usual, $|\mathsf g|:=\det(\mathsf g_{\ell m})$).

As a consequence of~\eqref{+++},~\eqref{+++1}, and~\eqref{+++2} for
some $\vartheta<\pi/2$ and some positive $\delta$ and $\gamma$ we
obtain
\begin{equation}\label{+++3}
|\arg(\mathsf q_\lambda[u,u]+\gamma\|u\|^2)|\leq \vartheta,\quad
\delta(du,du)_0-\gamma\|u\|^2\leq \Re\mathsf q_\lambda[u,u]\leq
\bigl((du,du)_0+\|u\|^2\bigr)/\delta
\end{equation}
uniformly in $\lambda\in\mathcal O_\alpha$. Therefore $\mathsf
q_\lambda$ in $L^2(\mathcal M)$ with domain $H^1(\mathcal M)$ is a
closed densely defined sectorial form and its sector $\{\mu\in\Bbb
C: |\arg(\mu+\gamma)|\leq\vartheta\}$ is independent of $\lambda$.
By construction the function $\mathcal
O_\alpha\ni\lambda\mapsto\mathsf q_\lambda[u,u]$ is analytic for any
$u\in H^1(\mathcal M)$.
\end{proof}

\section{Conjugated operator and its essential spectrum}\label{s5}

In order to show exponential decay of outgoing (resp. incoming)
solutions in infinite PMLs we study the operator
${^\lambda\!\Delta}$, $\Im\lambda>0$ (resp. $\Im\lambda<0$),
conjugated with an exponent.

Let  $\mathsf s$ be a smooth function on the semi-cylinder $\Pi$,
which depends only on the axial variable $x\in\Bbb R_+$ and
possesses  the properties~\eqref{ab}. We extend $\mathsf s$ to a
smooth function on $\mathcal M$ by setting $\mathsf
s\!\upharpoonright_{\mathcal M\setminus\Pi}\equiv 0$. Consider the
conjugated operator ${^\lambda\!\Delta}_\beta=e^{-\beta\mathsf
s}\,{^\lambda\!\Delta}\,e^{\beta\mathsf  s}$ with parameter
$\beta\in\Bbb C$ on functions  $\{u\in C_c^\infty(\mathcal M):
e^{\beta\mathsf  s} u\in \mathscr D({^\lambda\!\Delta})\}$, where
$e^{\beta\mathsf s}$ is the operator of multiplication by the
exponent. With ${^\lambda\!\Delta}_\beta$ we associate the quadratic
form
$$
\mathsf q_\lambda^\beta[u,u]=\bigl( d(e^{\beta \mathsf s}
u),d(\overline{e^{-\beta \mathsf s}\varrho_\lambda}u)\bigr)_\lambda,
\quad u\in C_c^\infty(\mathcal M).
$$

\begin{lemma}\label{q beta}
For any $\lambda\in\mathcal O_\alpha$ the form $\mathsf
q_\lambda^\beta$  in $L^2(\mathcal M)$  is a closed sectorial form
with the domain $H^1(\mathcal M)$. Moreover, its sector is
independent of $\lambda$.
\end{lemma}
\begin{proof} We have
$$
\mathsf q_\lambda^\beta[u,u]-\mathsf q_\lambda[u,u]=\beta\bigl(u\,
d\mathsf s,d(\varrho_{\overline{\lambda}}u)\bigr)_\lambda
-\beta\bigl(\varrho_\lambda\, du,u\,d\mathsf s\bigr)_\lambda
-\beta^2\bigl(\varrho_\lambda u\,d\mathsf s,u\,d\mathsf
s\bigr)_\lambda,
$$
where the right hand side depends linearly on $d u$. Similarly
to~\eqref{+++2} we conclude that the difference $\mathsf
q^\beta_\lambda- \mathsf q_\lambda$ has an arbitrarily small
relative bound  with respect to $\mathsf q_\lambda$. More precisely,
$$
|\mathsf q_\lambda^\beta[u,u]-\mathsf q_\lambda[u,u]|\leq
\varepsilon |\mathsf q_\lambda[u,u]|+ C(|\beta|,\varepsilon)\|u\|^2,
$$
 where $\varepsilon>0$ is arbitrarily small  and
$C(|\beta|,\varepsilon)$ is independent of $\lambda\in\mathcal
O_\alpha$ and $u\in C_c^\infty(\mathcal M)$. This together
with~\eqref{+++3} completes the proof.
 \end{proof}

 The Friedrichs
extension of ${^\lambda\!\Delta}_\beta$ is an m-sectorial
operator~\cite{Kato,Simon Reed iv}. Consider its domain $\mathscr
D({^\lambda\!\Delta}_\beta)$ as a Hilbert space with the norm
$\sqrt{\|\cdot\|^2+\|{^\lambda\!\Delta}_\beta\cdot\|^2}$. We say
that $\mu$ is a point of the essential spectrum
$\sigma_{ess}({^\lambda\!\Delta}_\beta)$ if the bounded operator
\begin{equation}\label{5}
{^\lambda\!\Delta}_\beta-\mu :\mathscr
D({^\lambda\!\Delta}_\beta)\to L^2(\mathcal M)
\end{equation}
is not  Fredholm. (Recall that  a bounded linear operator is said to
be  Fredholm, if  its kernel  and cokernel
 are finite-dimensional, and the
range  is closed.)
\begin{proposition}\label{ess}
Let $\lambda\in \mathcal O_\alpha$ and $\beta\in\Bbb C$.  Then
$\mu\in \sigma_{ess}({^\lambda\!\Delta}_\beta)$ if and only if
\begin{equation}\label{eq9}
\mu-(1+\lambda)^{-2}(\xi+i\beta)^2\in\sigma(\Delta_\Omega^N)
\end{equation}
for  some $\xi\in\Bbb R$, where $\sigma(\Delta_\Omega^N)$ is the
spectrum of the selfadjoint Neumann Laplacian $\Delta_\Omega^N$ on
$( \Omega,\mathfrak h)$.
\end{proposition}
The spectrum $\sigma_{ess}({^\lambda\!\Delta}_\beta)$ is depicted on
Fig.~\ref{fig5}. In the case $\beta=0$ the parabolas collapse to the
dashed rays  and we obtain the essential spectrum of
${^\lambda\!\Delta}\equiv{^\lambda\!\Delta}_0$.
\begin{figure}
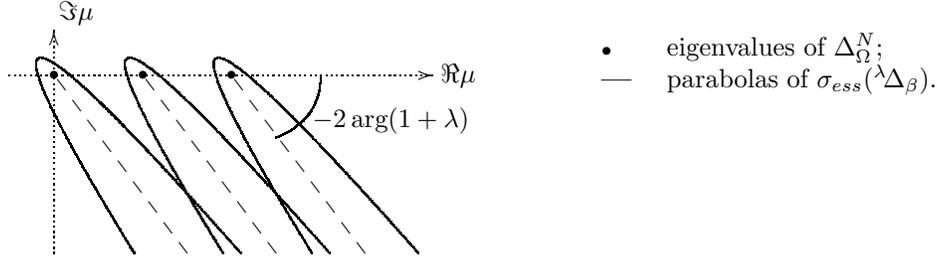
\centering
\[
\xy (0,0)*{\xy0;/r.28pc/:{\ar@{.>}(0,10);(50,10)*{\ \Re \mu}};
{\ar@{.>}(5,-10);(5,15)}; (7,17)*{\ \Im \mu};
{\ar@{--}(5,10)*{\scriptstyle\bullet};(20,-10)};(14,-10);(26,-10)**\crv{(-13,34)};
{\ar@{--}(15,10)*{\scriptstyle\bullet};(30,-10)};(24,-10);(36,-10)**\crv{(-3,34)};
{(25,10)*{\scriptstyle\bullet};(40,-10)*{} **\dir{--}};
(35,10);(30,3)**\crv{(35,5)*{\quad\quad\quad\quad\quad\
-2\arg(1+\lambda)}};(34,-10);(46,-10)**\crv{(7,34)};
\endxy};
(70,0)*{\xy (20,20)*{
\begin{array}{ll}
     {\scriptstyle\bullet } & \text{ eigenvalues of $\Delta^N_\Omega$;} \\
    \text{\bf---} & \text{ parabolas of $\sigma_{ess}({^\lambda\!\Delta}_\beta)$.}
\end{array}};
\endxy};
\endxy
\]
\caption{Essential spectrum of the operator
${^\lambda\!\Delta}_\beta$ for $\Im\lambda>0$ and $\beta\gtrless 0$,
the dashed rays correspond to $\beta = 0$.}\label{fig5}
\end{figure}
The proof of Proposition~\ref{ess} is preceded by
\begin{lemma}\label{lem} The operator ${^\lambda\!\Delta}$, $\lambda\in\mathcal O_\alpha$,
corresponds to a regular  elliptic boundary value problem on
$\mathcal M$.
\end{lemma}
\begin{proof} On $\mathcal M\setminus({r},\infty)\times\Omega$ the operator ${^\lambda\!\Delta}$
coincides with the  Neumann  Laplacian, which corresponds to a
regular elliptic problem~\cite{Lions Magenes,KozlovMazyaRossmann}.

 Let $y\in\Bbb R^n$ be a  system of
coordinates in a neighborhood $\mathscr U_j$ on $\Omega$. In the
case $\partial\Omega\cap\mathscr U_j\neq\varnothing$ we pick
boundary normal coordinates $y=(y',y_n)$ such that $\partial_{y_n}$
coincides with the unit inward normal derivative given by the metric
$\mathfrak h$.  From~\eqref{LF} it follows that  in the coordinates
$(x,y)$ we have
\begin{equation}\label{oper}
{^\lambda\!\Delta}=-|\mathsf g_\lambda|^{-1/2}\partial_\ell|\mathsf
g_\lambda|^{1/2}\mathsf g_\lambda^{\ell m}\partial_m
\end{equation}
and  the functions $u\in \mathscr D({^\lambda\!\Delta})$ satisfy
some boundary condition $^\lambda\!\mathcal N u=0$ on
$\partial\mathcal M$ such that
\begin{equation}\label{bc}
^\lambda\!\mathcal N u=\mathsf g^{n m}_\lambda\partial_m
u\upharpoonright_{y_n=0}=0
\end{equation} if $\partial\Omega\cap\mathscr U_j\neq\varnothing$; here
notations are the same as in the proof of Lemma~\ref{L1}. The
operator~\eqref{oper} is strongly elliptic for $x\geq {r}$ due
to~\eqref{++}. It remains to show that the Shapiro-Lopatinski\v{\i}
condition is met on $({r},\infty)\times(\partial\Omega\cap\mathscr
U_j)$. Thanks to~\eqref{STAB} the principal part of~\eqref{oper} for
$y_n=0$ (resp. the principal part of the operator of boundary
conditions in~\eqref{bc}) is little different from $
-(1+\lambda\mathsf
s'_{r}(x))^{-2}\partial_0^2-Q(y',\partial_{y'})-\partial^2_n$ (resp.
from $\partial_n$) uniformly in $x\geq {r}$, $\lambda\in\mathcal
D_\alpha$, and $y'$ as ${r}$ is large. Here
$-Q(y',\partial_{y'})-\partial^2_n$ is the principal part of
$\Delta_\Omega$. For $\xi=(\xi_0,\xi')$ on the unit sphere $S^n$ and
every bounded solution $\mathsf u\neq 0$ of
$$
\bigl((1+\lambda \mathsf s'_{r}(x))^{-2}\xi_0^2+Q(y',
\xi')-\partial_{y_n}^2\bigr)\mathsf u(y_n)=0\text{ for } y_n\in\Bbb
R_+
$$
we have
$$
\partial_{y_n}\mathsf u(0)=C\partial_{y_n} e^{-y_n\sqrt{(1+\lambda \mathsf s'_{r}(x))^{-2}\xi_0^2+Q( y', \xi')}}\upharpoonright_{y_n=0}=-C\sqrt{(1+\lambda \mathsf s'_{r}(x))^{-2}\xi_0^2+Q(y',\xi')}\neq
0
$$
because $\Re (1+\lambda \mathsf s'_{r}(x))^{-2}>0$ and $Q(y',\xi')$
is a positive definite quadratic form in $\xi'$. In other words, the
pair $\bigl(-(1+\lambda\mathsf
s'_{r}(x))^{-2}\partial_0^2+Q(y',\partial_{y'})-\partial^2_n,\partial_n\bigr)$
satisfies the Shapiro-Lopatinski\v{\i} condition or, equivalently,
the estimate
\begin{equation}\label{est}
\|\mathsf u\|_{ H^2(\Bbb R_+)}\leq C\bigl(\bigl\|\bigl((1+\lambda
\mathsf s'_{r}(x))^{-2}\xi_0^2+Q(
y',\xi')-\partial_{y_n}^2\bigr)\mathsf u\bigr\|_{ L^2(\Bbb
R_+)}+|\partial_{y_n}\mathsf u(0)|\bigr)
\end{equation}
holds, e.g.~\cite{KozlovMazyaRossmann, Lions Magenes}. The constant
$C$ in~\eqref{est} is independent of ${r}>0$, $\xi\in S^n$,
$\lambda\in\overline{\mathcal O_\alpha}$, $x\in\Bbb R_+$, and $y'\in
\mathscr U_j\cap\partial\Omega$. (Indeed, it suffices to note
that~\eqref{est} with ${r}$ replaced by  ${r}'$ can be obtained by
the change of variables $x\mapsto x+{r}-{r}'$, the function $\mathsf
s'_{r}$ varies over a compact subset of $\Bbb R_+$ only, and
$(1+\lambda \mathsf s'_{r}(x))^{-2}\xi_0^2+Q( y',\xi')$ is smooth in
$\lambda$, $\xi$, $x$, and $y'$.) This together with~\eqref{STAB}
implies that the estimate similar to~\eqref{est} is valid for the
principal parts of~\eqref{oper} and~\eqref{bc} provided $x\geq {r}$,
where ${r}$ is sufficiently large. This completes the proof.
\end{proof}

Now we are in position to prove Proposition~\ref{ess}.

\begin{proof} We will rely on the following  lemma  due to Peetre, see
e.g.~\cite[Lemma~5.1]{Lions Magenes} or
\cite[Lemma~3.4.1]{KozlovMazyaRossmann}:
\begin{itemize}
\item[]
{\it Let $\mathcal X,\mathcal Y$ and $\mathcal Z$ be Banach spaces,
where $\mathcal X$ is compactly embedded into $\mathcal Z$.
Furthermore, let $\mathcal L$ be a linear continuous operator from
$\mathcal X$ to $\mathcal Y$. Then the next two assertions are
equivalent: (i) the range of $\mathcal L$ is closed in $\mathcal Y$
and $\dim \ker \mathcal L<\infty$, (ii) there exists a constant $C$
such that
    \begin{equation}\label{coercive}
    \|u\|_{\mathcal X}\leq C(\|\mathcal L u\|_{\mathcal Y}+\|u\|_{\mathcal Z})\quad \forall u\in \mathcal X.
    \end{equation}}
\end{itemize}

{\it Sufficiency.} Here we assume that  $\mu$ does not
satisfy~\eqref{eq9} and establish an estimate of
type~\eqref{coercive} for the operator~\eqref{5}. Consider the
operator $({^\lambda\!\Delta}, {^\lambda\!\mathcal N}):
C_c^\infty(\mathcal M)\to C_c^\infty(\mathcal M)\times\
C_c^\infty(\partial\mathcal M)$  satisfying
$$
({^\lambda\!\Delta} u,v)_0+\langle {^\lambda\!\mathcal N}u,
v\rangle=\mathsf q_\lambda[u,v],\quad u\in C_c^\infty(\mathcal
M),v\in H^1(\mathcal M),
$$
where ${^\lambda\!\mathcal N}$ is the operator of boundary
conditions on $\partial\mathcal M$ and  $\langle\cdot,\cdot\rangle$
is the inner product on $(\partial\mathcal M,\mathsf
g\!\upharpoonright_{\partial\mathcal M})$. Clearly,
${^\lambda\!\mathcal N}u=0$ for $u\in\mathscr
D({^\lambda\!\Delta})$. Besides, in the local coordinates $(x,y)$ on
$\Bbb R_+\times\Omega$ we have~\eqref{oper} and~\eqref{bc}. In
particular, ${^0\!\Delta}$ is the Laplace-Beltrami operator  on
$(\mathcal M,\mathsf g)$ and $^0\!\mathcal N$ is the corresponding
operator of the Neumann boundary conditions. Consider also the
 conjugated operator $({^\lambda\!\Delta_\beta},{^\lambda\!\mathcal N_\beta})=e^{-\mathsf s \beta}({^\lambda\!\Delta}, {^\lambda\!\mathcal N})e^{\mathsf s
 \beta}$.

Introduce the Sobolev space $H^\ell(\Bbb R\times\Omega)$ as the
completion of the set $C_c^\infty(\Bbb R\times\Omega)$ with respect
to the norm
$$
\|u\|_{H^\ell(\Bbb R\times\Omega)}=\Bigl (\int_{\Bbb R}\sum_{k\leq
\ell}\|\partial^k_x u(x)\|^2_{ H^{\ell-k}(\Omega)}\,dx\Bigr)^{1/2},
$$
where $H^\ell(\Omega)$ is the Sobolev space on $\Omega$. Let
$H^{1/2}(\Bbb R\times\partial\Omega)$ be the space of traces
$u\!\upharpoonright_{\Bbb R\times\partial\Omega}$ of the functions
$u\in H^1(\Bbb R\times\Omega)$.   Denote by $\partial_\eta$ the
operator of Neumann boundary conditions on $\Bbb
R\times\partial\Omega$ taken with respect to the metric   $dx\otimes
dx+\mathfrak h$. Applying the Fourier transform $\mathcal
F_{x\mapsto \xi}$ we pass from the continuous operator
\begin{equation}\label{opo}
\bigl(\Delta_\Omega-(1+\lambda)^{-2}(\partial_x+\beta)^2-\mu,\partial_\eta\bigr):
H^2(\Bbb R\times\Omega)\to H^0(\Bbb R\times\Omega)\times
H^{1/2}(\Bbb R\times\partial\Omega)
\end{equation}
 to the operator $\bigl(\Delta_\Omega+(1+\lambda)^{-2}(\xi+i\beta)^2-\mu,\partial_\eta\bigr)$
of the Neumann  problem  on  $(\Omega,\mathfrak h)$. As $\mu$ does
not meet~\eqref{eq9}, the latter operator is invertible for all
$\xi$ and the inverse of~\eqref{opo} is given by
$$
\mathcal F^{-1}_{\xi\mapsto x}
\bigl(\Delta_\Omega+(1+\lambda)^{-2}(\xi+i\beta)^2-\mu,\partial_\eta\bigr)^{-1}\mathcal
F_{x\mapsto \xi};
$$
 see e.g.~\cite[Theorem 5.2.2]{KozlovMazyaRossmann} or
\cite[Theorem~2.4.1]{KozlovMaz`ya} for details. As a consequence we
have
\begin{equation}\label{op}
\|u\|_{ H^2(\Bbb R\times\Omega)}\leq
C\bigl\|\bigr(\Delta_\Omega-(1+\lambda)^{-2}(\partial_x+\beta)^2-\mu,\partial_\eta
\bigr)u\bigr\|_{H^0(\Bbb R\times\Omega)\times H^{1/2}(\Bbb
R\times\partial\Omega)}.
\end{equation}

 Let $\chi_T(x)=\chi(x-T)$, where $\chi\in  C^\infty(\Bbb R)$ is a cutoff function such that
$\chi(x)=1$ for $x\geq 1$ and $\chi(x)=0$ for $x\leq 0$. Then due
to~\eqref{STAB-},~\eqref{oper}, and~\eqref{bc} the constant $c(T)$
in the estimate
$$
\bigl\|\bigl({^\lambda\!\Delta_\beta}-\Delta_\Omega+(1
+\lambda)^{-2}(\partial_x+\beta)^2,{^\lambda\!\mathcal
N_\beta-\partial_\eta\bigr)\chi_T u }\|_{ H^0(\Bbb
R\times\Omega)\times H^{1/2}(\Bbb R\times\partial\Omega)}\leq
c(T)\|\chi_T u\|_{H^2(\Bbb R\times\Omega)}
$$
 tends to zero as
$T\to+\infty$. This together with~\eqref{op} implies that for all
sufficiently large $T$ the estimate
\begin{equation}\label{einf}
 \|\chi_Tu\|_{ H^{2}(\Bbb R\times\Omega)} \leq \texttt{C} \|({^\lambda\!\Delta_\beta} -  \mu,{^\lambda\!\mathcal N_\beta}) \chi_T u\|_{H^0(\Bbb R\times\Omega)\times H^{1/2}(\Bbb
R\times\partial\Omega)}
\end{equation}
  holds, where  $\texttt{C}=(1/C- c(T))^{-1}>0$.

Let $H^\ell(\mathcal M)$ be the Sobolev space introduced as the
completion of the set $C_c^\infty(\mathcal M)$ in the norm
$$
\| u\|_{H^\ell(\mathcal M)}=\Bigl(\|u\|^2_{H^\ell(\mathcal
M_c)}+\int_{\Bbb R_+}\sum_{k\leq \ell}\|\partial_x^k
u(x)\|^2_{H^{\ell-k}(\Omega)}\,dx\Bigr)^{1/2},
$$
where $\partial_x=\partial/\partial x$ and $H^\ell(\mathcal M_c)$ is
the Sobolev space on the compact manifold $\mathcal M_c$. By
$H^{1/2}(\partial\mathcal M)$ we denote the space of traces on
$\partial \mathcal M$ of functions in $H^1(\mathcal M)$. We extend
$\chi_T$ from its support in $\Pi$ to $\mathcal M$ by zero and
rewrite~\eqref{einf} in the form
$$
\|\chi_Tu\|_{ H^{2}(\mathcal M)} \leq  \texttt{C}
\bigl(\|\chi_T(^\lambda\!\Delta_\beta - \mu,{^\lambda\!\mathcal
N_\beta})u\|_{L^2(\mathcal M)\times H^{1/2}(\mathcal M)}
+\|[(^\lambda\!\Delta_\beta,{^\lambda\!\mathcal
N_\beta}),\chi_T]u\|_{L^2(\mathcal M)\times H^{1/2}(\mathcal M)}.
$$
For the commutator we have
$[({^\lambda\!\Delta_\beta},{^\lambda\!\mathcal
N_\beta}),\chi_T]u=[({^\lambda\!\Delta_\beta},{^\lambda\!\mathcal
N_\beta}),\chi_T](1-\chi_{2T})u$ and therefore
$$
\|[({^\lambda\!\Delta_\beta},{^\lambda\!\mathcal
N_\beta}),\chi_T]u\|_{ L^2(\mathcal M)\times H^{1/2}(\mathcal
M)}\leq C \|(1-\chi_{2T}) u\|_{H^2(\mathcal M)}.
$$
Moreover, as a consequence of Lemma~\ref{lem} the  local elliptic
coercive estimate
$$
\|(1-\chi_{2T})u\|_{H^2(\mathcal M)}\leq
C\bigl(\|(1-\chi_{3T})({^\lambda\!\Delta_\beta}-\mu
,{^\lambda\!\mathcal N_\beta})u\|_{L^2(\mathcal M)\times
H^{1/2}(\partial\mathcal M)}+\|(1-\chi_{3T}) u\|\bigr)
$$
is valid~\cite{Lions Magenes}. From the last three estimates it
follows that
\begin{equation}\label{coercive!}
\| u\|_{H^2(\mathcal M)}\leq C\bigl(\|({^\lambda\!\Delta_\beta}-\mu
,{^\lambda\!\mathcal N_\beta})u\|_{L^2(\mathcal M)\times
H^{1/2}(\partial\mathcal M)}+\|(1-\chi_{3T}) u\|\bigr).
\end{equation}
In particular,~\eqref{coercive!} implies that
$\|\cdot\|_{H^2(\mathcal M)}$ is an equivalent norm in  $\mathscr
D({^\lambda\!\Delta_\beta})=\{u\in H^2(\mathcal
M):{^\lambda\!\mathcal N_\beta}u=0 \}$.

Let $\mathsf w$ be a bounded rapidly decreasing at infinity positive
function on $\mathcal M$ such that the embedding of $H^2(\mathcal
M)$ into the weighted space $L^2(\mathcal M,\mathsf w)$ with the
norm $\|\mathsf w \cdot\|$ is compact. Then $\mathscr
D({^\lambda\!\Delta_\beta})$ is compactly embedded into
$L^2(\mathcal M,\mathsf w)$ and~\eqref{coercive!} is an estimate of
type~\eqref{coercive} for the operator~\eqref{5}. Thus the range
of~\eqref{5} is closed and the kernel is finite-dimensional. In
order to see that the cokernel is finite-dimensional, one can apply
the same argument and obtain an estimate of type~\eqref{coercive}
for the adjoint operator.

{\it Necessity.} Now we assume that $\mu$ meets~\eqref{eq9} for some
$j$ and show that the operator~\eqref{5} is not Fredholm. It
suffices to find a sequence $\{v_\ell\}_{\ell=1}^\infty$ in
$\mathscr D({^\lambda\!\Delta})$ violating the
estimate~\eqref{coercive!}.

We first show that for a regular point $\mu_0$ of the m-sectorial
operator ${^\lambda\!\Delta_\beta}$ the continuous operator
\begin{equation}\label{ad}
 ({^\lambda\!\Delta_\beta} -  \mu_0,{^\lambda\!\mathcal N_\beta}): H^2(\mathcal M)\to L^2(\mathcal M)\times H^{1/2}(\partial\mathcal M)
\end{equation}
realizes an isomorphism. With this aim in mind we  replace
$(1-\chi_{3T}) u$ by $\mathsf w u$ in~\eqref{coercive!}.  Then  by
the Peetre lemma  the range of~\eqref{ad} is closed. It is easy to
see that the elements in the cokernel of the operator~\eqref{ad} are
of the form $(v, v\!\upharpoonright_{\partial\mathcal M})$ with
$v\in\ker({^\lambda\!\Delta_\beta^*} - \mu_0)$, where
${^\lambda\!\Delta_\beta^*}$ is  adjoint to the m-sectorial operator
${^\lambda\!\Delta_\beta}$  in $L^2(\mathcal M)$ with domain
$\mathscr D({^\lambda\!\Delta_\beta})$. Indeed, let $(v,\underline
v)$ be in the kernel of the adjoint operator
$$ ({^\lambda\!\Delta_\beta} -  \mu_0,{^\lambda\!\mathcal N_\beta})^*: L^2(\mathcal M)\times (H^{1/2}(\partial\mathcal M))^* \to (H^{2}(\mathcal M))^*.
$$
Then $({^\lambda\!\Delta_\beta}u-\mu_0 u,v)_0+\langle
{^\lambda\!\mathcal N_\beta} u, \underline v\rangle=0$ for all $u\in
H^2(\mathcal M)$, where $\langle\cdot,\cdot\rangle$ is extended to
$H^{1/2}(\partial\mathcal M)\times (H^{1/2}(\partial\mathcal M))^*$.
Since $\mathscr D({^\lambda\!\Delta_\beta})\subset H^2(\mathcal M)$,
we immediately see that $v\in \ker({^\lambda\!\Delta_\beta^*} -
\mu_0)$. Then for $u\in H^2(\mathcal M)$ the Green identity gives $
\langle {^\lambda\!\mathcal N_\beta} u, \underline v -
v\!\upharpoonright_{\partial\mathcal M}\rangle=0$ and therefore
$\underline v = v\!\upharpoonright_{\partial\mathcal M}$. As a
consequence,~\eqref{ad} is an isomorphism for
$\mu_0\notin\sigma({^\lambda\!\Delta_\beta})$.

Let $\chi\in C^\infty(\Bbb R)$ be  such that $\chi(x)=1$ for
$|x-3|\leq 1$, and $\chi(x)=0$ for $|x-3|\geq 2$. Consider the
functions
\begin{equation}\label{test}
u_\ell(x,
y)=\chi(x/\ell)\exp\bigl({i(1+\lambda)\sqrt{\mu-\nu_j}x}\bigr)\Phi(
y),\quad (x,y)\in \Bbb R\times\Omega,
\end{equation}
where $\Phi$ is an eigenfunction of the Neumann Laplacian
$\Delta^N_\Omega$ corresponding to the eigenvalue $\nu_j$. It is
clear that $u_\ell$ satisfies the Neumann boundary condition
$\partial_\eta u_\ell=0$ on $\Bbb R\times\partial\Omega$. As $\mu$
meets the condition~\eqref{eq9}, the exponent in~\eqref{test} is an
oscillating function of $x\in\Bbb R$. Straightforward calculation
shows that
\begin{equation}\label{s--}
\bigl\|\bigl(\Delta_\Omega-(1+  \lambda)^{-2}(\partial_x+\beta)^2
-\mu\bigr) u_\ell\bigr\|_ {H^0(\Bbb R\times\Omega)}\leq
const,\quad\|u_\ell\|_{ H^2(\Bbb R\times\Omega)}\to\infty
\end{equation}
as $\ell\to +\infty$. We extend the functions $u_\ell$ from $\Pi$ to
$\mathcal M$ by zero and set
$$
v_\ell=u_\ell-({^\lambda\!\Delta_\beta} - \mu_0,{^\lambda\!\mathcal
N_\beta})^{-1}(0,{^\lambda\!\mathcal N_\beta} u_\ell),
$$
where $\mu_0$ is a regular point of the m-sectorial operator
${^\lambda\!\Delta_\beta}$. Clearly, $v_\ell\in\mathscr
D({^\lambda\!\Delta_\beta})$. We also have
\begin{equation}\label{s-}
\|({^\lambda\!\Delta_\beta} - \mu_0,{^\lambda\!\mathcal
N_\beta})^{-1}(0,{^\lambda\!\mathcal N_\beta}
u_\ell)\|_{H^2(\mathcal M)} \leq C \|({^\lambda\!\mathcal
N_\beta}-\partial_\eta)u_\ell\|_{H^{1/2}(\partial\mathcal M)}\leq
\texttt{C}_\ell\|u_\ell\|_{ H^2(\Bbb R\times\Omega)},
\end{equation}
where $\texttt{C}_\ell\to 0$ as $\ell\to+\infty$. Hence
\begin{equation}\label{s+}
\|v_\ell\|_{H^2(\mathcal M)}\geq \|u_\ell\|_{ H^2(\Bbb
R\times\Omega)}-\texttt{C}_\ell\|u_\ell\|_{ H^2(\Bbb
R\times\Omega)},\quad \|v_\ell-u_\ell\|_{H^2(\mathcal M)}\leq
\texttt{C}_\ell\|u_\ell\|_{ H^2(\Bbb R\times\Omega)}.
\end{equation}
Assume that the estimate~\eqref{coercive!} is valid. Without loss of
generality we can take a rapidly decreasing weight $\mathsf w$ such
that $\|\mathsf w u_\ell\|\leq Const$ uniformly in $\ell$ and the
embedding $H^2(\mathcal M) \hookrightarrow L^2(\mathcal M;\mathsf
w)$ is compact. It is clear that $\|\mathsf w u\|\leq
c\|u\|_{H^2(\mathcal M)}$ with some independent of $u\in
H^2(\mathcal M)$ constant $c$. Therefore
\begin{equation}\label{s.}
\|\mathsf w v_\ell\|\leq Const + c \texttt{C}_\ell\|u_\ell\|_{
H^2(\Bbb R\times\Omega)},
\end{equation}
cf.~\eqref{s+}. The estimate
$$
\|({^\lambda\!\Delta_\beta}-\Delta_\Omega+(1+
\lambda)^{-2}(\partial_x+\beta)^2)u_\ell\|\leq\texttt{c}_\ell\|u_\ell\|_{H^2(\Bbb
R\times\Omega)},
$$
 where $\texttt{c}_\ell\to 0$ as $\ell\to+\infty$, together with~\eqref{s--} and~\eqref{s+}  gives
\begin{equation}\label{s++}
\begin{aligned}
\|({^\lambda\!\Delta_\beta}-\mu)v_\ell\|\leq C\bigl\|\bigl(-(1+
\lambda)^{-2}(\partial_x+\beta)^2+\Delta_\Omega
-\mu\bigr) u_\ell\bigr\|_ {H^0(\Bbb R\times\Omega)}\\
+\|({^\lambda\!\Delta_\beta}-\Delta_\Omega+(1+
\lambda)^{-2}(\partial_x+\beta)^2)u_\ell\|+\|({^\lambda\!\Delta_\beta}-\mu)(v_\ell-u_\ell)\|
\\
\leq C\cdot const+(\texttt{c}_\ell+\mathsf{C}
\texttt{C}_\ell)\|u_\ell\|_{H^2(\Bbb R\times\Omega)}.
\end{aligned}
\end{equation}
Finally, as a consequence
of~\eqref{s+},~\eqref{coercive!},~\eqref{s++} and~\eqref{s.},  we
get
\begin{equation}\label{final}
\begin{aligned}
\|u_\ell\|_{ H^2(\Bbb R\times\Omega)}-\texttt{C}_\ell\|u_\ell\|_{
H^2(\Bbb R\times\Omega)} \leq\|v_\ell\|_{H^2(\mathcal M)}\leq
\mathrm C(\|({^\lambda\!\Delta_\beta}-\mu)v_\ell\|+\|\mathsf w
v_\ell\|)
\\
\leq \mathrm C( C\cdot const +(\mathsf
C\texttt{C}_\ell+\texttt{c}_\ell)\|u_\ell\|_{ H^2(\Bbb
R\times\Omega)}+Const + c \texttt{C}_\ell\|u_\ell\|_{ H^2(\Bbb
R\times\Omega)}).
\end{aligned}
\end{equation}
Since $\texttt{c}_\ell\to 0$ and $\texttt{C}_\ell\to 0$, the
inequalities~\eqref{final} imply that the value $\|u_\ell\|_{
H^2(\Bbb R\times\Omega)}$ remains bounded as $\ell\to+\infty$. This
contradicts~\eqref{s--}. Thus the sequence
$\{v_\ell\}_{\ell=1}^\infty$ violates the
estimate~\eqref{coercive!}. \end{proof}

\section{Exponential decay of solutions in infinite PMLs}\label{s6}

 Consider the algebra  $\mathscr E$ of all entire functions
 $\mathbb C\ni z\mapsto F(z,\cdot)\in C^\infty(\Omega)$ with the following property:
in any sector $|\Im z|\leq (1-\epsilon) \Re z$ with $\epsilon>0$ the
value $\|F(z,\cdot)\|_{L^2(\Omega)}$ decays faster than any inverse
power of $\Re z$  as $\Re z\to+\infty$.  Examples of functions $F\in
\mathscr E$ are $F(z,y)=e^{-\gamma z^2}P(z,y)$, where  $\gamma>0$
and $P(z,y)$ is an arbitrary polynomial  in $z$ with coefficients in
$C^\infty(\Omega)$. We say that $f\in L^2(\mathcal M)$ is an
analytic vector, if $f(x,y)=F(x,y)$  for some $F\in\mathscr E$ and
all $(x,y)\in\Pi$. For any $f$ in the set $\mathcal A$ of all
analytic vectors  we can define the analytic function $\mathcal
O_\alpha\ni\lambda\mapsto f_\lambda\in L^2(\mathcal M)$ by setting
$f_\lambda=f$ on $\mathcal M\setminus ({r},\infty)\times\Omega$ and
$f_\lambda(x,y)=f(x+\lambda \mathsf s_{r}(x),y)$ for $(x,y)\in
({r},\infty)\times\Omega$. The set $\{f_\lambda: f\in\mathcal A\}$
is dense in $L^2(\mathcal M)$ for any $\lambda\in\mathcal D_\alpha$,
see, e.g.,~\cite[Theorem 3]{Hunziker}.

\begin{theorem}\label{T1} Assume that $\mu_0\in \Bbb R\setminus
\sigma(\Delta_\Omega^N)$ is not an eigenvalue (i.e. not a trapped
mode) of the self-adjoint Neumann Laplacian $^0\!\Delta$ in
$L^2(\mathcal M)$.  Then the following assertions are valid.
\begin{itemize}
\item[1.] For any $f\in \mathcal A$ there exist outgoing $u_-$ and
incoming $u_+$ solutions defined by the limiting absorption
principle
\begin{equation}\label{LAP}
u_+=\lim_{\epsilon \uparrow 0}({^0\!\Delta}-\mu_0-i\epsilon)^{-1}f,
\quad u_-=\lim_{\epsilon \downarrow
0}({^0\!\Delta}-\mu_0-i\epsilon)^{-1}f,
\end{equation}
where the limits are taken in the space $L^2_{loc}(\mathcal M)$.

\item[2. ] For $\lambda\in\mathcal O_\alpha$ with $\Im\lambda<0$ (resp. $\Im\lambda>0$)
the m-sectorial operator ${^\lambda\!\Delta}$ in $L^2(\mathcal M)$
 models PMLs on $({r},\infty)\times\Omega$ for incoming (resp. outgoing) solutions in the sense
that $u_\lambda=({^\lambda\!\Delta}-\mu_0)^{-1} f_\lambda$ coincides
on $\mathcal M\setminus({r},\infty)\times\Omega$ with the incoming
solution $u_+$
 (resp. with the outgoing solution $u_-$) and
\begin{equation}\label{interval}
\|e^{\beta\mathsf s}u_\lambda\|_{H^2(\mathcal M)}\leq
C(\mu_0,\lambda)\|e^{\beta\mathsf s} f_\lambda\| ,\quad
0\leq\beta<min_{\nu\in\sigma(\Delta^N_\Omega)}|\Im\{
(1+\lambda)\sqrt{\mu_0-\nu}\}|.
\end{equation}
The estimate~\eqref{interval} shows that $u_\lambda$ decays
exponentially in the PMLs.
\end{itemize}
\end{theorem}
\begin{proof} 1. We need to prove that for any $\varrho\in C_c^\infty(\mathcal
M)$ and $f\in\mathcal  A$  the function
$\varrho({^0\!\Delta}-\mu_0-i\epsilon)^{-1}f$ tends to some limits
in $L^2(\mathcal M)$ as $\epsilon \downarrow 0$ and $\epsilon
\uparrow 0$. Take a sufficiently large ${r}$ such that $\supp
\varrho\cap ({r},\infty)\times\Omega=\varnothing$. Let $\mu$ be
outside the sector of m-sectorial operator ${^\lambda\!\Delta}$.
Then for any real $\lambda\in\mathcal O_\alpha$ the change of
variable $x\mapsto x+ \lambda\mathsf s_{r}(x)$ outside $\supp
\varrho$ implies
\begin{equation}\label{z}
\varrho({^0\!\Delta}-\mu)^{-1}f=\varrho
({^\lambda\!\Delta}-\mu)^{-1}f_\lambda.
\end{equation}
From Lemma~\ref{L1} and~\eqref{LF} it follows that the resolvent
$({^\lambda\!\Delta}-\mu)^{-1}$ is an analytic function of
$\lambda\in\mathcal O_\alpha$ with values in the space of bounded
operators in $L^2(\mathcal M)$~\cite[Theorem XII.7]{Simon Reed iv}.
Thus~\eqref{z} extends by analyticity to all $\lambda\in\mathcal
O_\alpha$. As $\mu$ is a regular point of ${^\lambda\!\Delta}$ in
the the simply connected set $\Bbb
C\setminus\sigma_{ess}({^\lambda\!\Delta})$ (see
Proposition~\ref{ess}), the Fredholm analytic theory implies that
the resolvent
$$
\Bbb C\setminus\sigma_{ess}({^\lambda\!\Delta})\ni\mu\mapsto
({^\lambda\!\Delta}-\mu)^{-1}: L^2(\mathcal M)\to \mathscr
D({^\lambda\!\Delta})
$$
is a meromorphic operator function; see,
e.g.,~\cite[Appendix]{KozlovMaz`ya}. It remains to show that
$\mu_0\in \Bbb C\setminus\sigma_{ess}({^\lambda\!\Delta})$ is not a
pole. Then the right hand side of~\eqref{z} with $\Im\lambda<0$
(resp. $\Im\lambda>0$) provides the left hand side with analytic
continuation from $\mu=\mu_0+i\epsilon$ to $\mu=\mu_0$ as
$\epsilon\downarrow 0$  (resp. $\epsilon \uparrow 0$).

For $\mu$ outside of the sector of ${^\lambda\!\Delta}$ and real
$\lambda\in\mathcal O_\alpha$ by the change of variable $x\mapsto x+
\lambda\mathsf s_{r}(x)$ we obtain
\begin{equation}\label{abc}
\bigl(({^0\!\Delta}-\mu)^{-1}f,g\bigr)_0=\bigl(({^\lambda\!\Delta}-\mu)^{-1}f_\lambda,g_{\bar\lambda}\bigr)_{\lambda},
\quad f,g\in\mathcal A.
\end{equation}
This equality extends by analyticity to all $\lambda\in\mathcal
O_\alpha$.  Suppose, by contradiction, that $\mu_0$ is a pole of
$({^\lambda\!\Delta}-\mu)^{-1}$ with $\Im\lambda\lessgtr 0$. As the
sets $\{f_\lambda: f\in\mathcal A\}$ and
$\{g_{\bar\lambda}:g\in\mathcal A\}$ are dense in $L^2(\mathcal M)$,
the right hand side of~\eqref{abc} has a pole at $\mu_0$ for some
$f$ and $g$. Then~\eqref{abc} implies that $(\mathsf P_{\mu_0}
f,g)\neq 0$, where $\mathsf P_{\mu_0}$ is the projection onto the
eigenspace of the selfadjoint operator $^0\!\Delta$, and thus
$\ker(\Delta-\mu_0)\neq\{0\}$. This is a contradiction.

2. Since~\eqref{z} is valid for any $\varrho\in C_c^\infty(\mathcal
M)$ such that
$\supp\varrho\cap({r},\infty)\times\Omega=\varnothing$, our
construction in the proof of assertion 1 shows that $u_\lambda$
coincides on $\mathcal M\setminus({r},\infty)\times\Omega$ with
$u_+$
 (resp.  $u_-$) if $\Im\lambda<0$ (resp.
 $\Im\lambda>0$). The condition on $\beta$ in~\eqref{interval}
guarantees that $\mu_0$ together with all points outside of the
sector of m-sectorial operator ${^\lambda\!\Delta_\beta}$
 is in the simply connected subset of $\Bbb
 C\setminus\sigma_{ess}({^\lambda\!\Delta_\beta})$, see
 Proposition~\ref{ess}. Then the Fredholm analytic theory implies
that $\mu_0$ is either an eigenvalue of ${^\lambda\!\Delta_\beta}$
or $\mu_0\notin\sigma({^\lambda\!\Delta_\beta})$,
e.g.,~\cite[Appendix]{KozlovMaz`ya}. The inclusion $\Psi\in \ker
({^\lambda\!\Delta_\beta}-\mu_0)$ gives $e^{-\beta\mathsf
s}\Psi\in\ker ({^\lambda\!\Delta} -\mu_0)=\{0\}$ as $\mu_0$ is not a
pole of $({^\lambda\!\Delta} -\mu)^{-1}$. Hence
$\mu_0\notin\sigma({^\lambda\!\Delta_\beta})$. This together with
the equality $e^{\beta\mathsf s}u_\lambda=({^\lambda\!\Delta_\beta}
-\mu_0)^{-1}e^{\beta \mathsf s} f_\lambda$ justifies the
estimate~\eqref{interval}, cf.~\eqref{coercive!}.
\end{proof}

\section{Finite PMLs, stability and exponential convergence of the PML
method}\label{s7} Consider the compact manifold $\mathcal
M_{R}=\mathcal M\setminus ({R},\infty)\times\Omega$. The boundary
$\partial \mathcal M_{R}$ of $\mathcal M_{R}$ is piecewise smooth.
It has two conic points
 in the case of a 1-dimensional manifold
$\Omega$ and the edge $\partial\Omega\times\{{R}\}$ otherwise. We
denote $\partial \mathcal
M_{R}\setminus(\{{R}\}\times\Omega)=\partial \mathcal M_{R}^-$.
Introduce the Sobolev space $H^2(\mathcal M_{R})$ as the completion
of the set $C_c^\infty(\mathcal M_{R})$ with respect to the norm
$$
\|u\|_{ H^2(\mathcal M_{R})}=\Bigl(\|u\|^2_{ H^2(\mathcal
M_c)}+\sum_{\ell\leq 2}\int_0^{R}\|\partial_x^\ell u(x)\|^2_{
H^{2-\ell}(\Omega)}\,dx\Bigr)^{1/2},$$ where $H^2(\mathcal M_c)$ and
$H^{2-\ell}(\Omega)$ are the Sobolev spaces on the smooth compact
manifolds $\mathcal M_c$ and $\Omega$. Consider the problem with
finite PMLs: {\it given $F\in L^2(\mathcal M_{R})$ find a solution
$v\in H^2(\mathcal M_{R})$ to the problem}
\begin{equation}\label{MR}
({^\lambda\!\Delta}-\mu_0)v=F \text{ on } \mathcal
M_{R};\quad{^\lambda\!\mathcal N} v=0 \text{ on } \partial \mathcal
M_{R}^-;\quad v=0 \text{ on } \{{R}\}\times\Omega.
\end{equation}
  In the next theorem we prove a stability result for this problem.
\begin{theorem}\label{thMR} Let $\lambda\in\mathcal O_\alpha\setminus\Bbb R$. Assume that  $\mu_0\in \Bbb
R\setminus\sigma(\Delta_\Omega)$ is not an eigenvalue of the
selfadjoint Neumann Laplacian $^0\!\Delta$ in $L^2(\mathcal M)$.
Then there exists a large number ${R}_0>0$ such that for all
${R}>{R}_0$ and $F\in L^2(\mathcal M_{R})$,~\eqref{MR} has a unique
solution $v\in H^2(\mathcal M_{R})$. Moreover, the estimate
\begin{equation}\label{estMR}
\|v\|_{ H^2(\mathcal M_{R})}\leq C \|F\|_{ L^2(\mathcal M_{R})}
\end{equation}
holds with independent of ${R}>{R}_0$ and $F$ constant $C$.
\end{theorem}
In the proof of Theorem~\ref{thMR} we rely on  compound expansions.
This requires construction of an approximate solution to~\eqref{MR}
compounded of solutions to limit problems,
e.g.,~\cite{KozlovMazyaRossmann}. Being substituted into~\eqref{MR}
the approximate solution leaves a discrepancy, which tends to zero
as ${R}$ increases. In contrast to the case of the Dirichlet
Laplacian~\cite{Kalvin SIMA}, here  the discrepancy left in the
boundary conditions on $\partial\mathcal M_{R}^-$ cannot be made
small for large ${R}$ if we use homogeneous limit problems. As the
first limit problem we take the problem with infinite PMLs and
non-homogeneous boundary conditions on $\partial \mathcal M$. In the
next lemma we study the second limit problem.

 Introduce the weighted Sobolev space $H^k_{\beta}((-\infty,{R})\times\Omega)$  as the completion of the set $C_c^\infty((-\infty,{R})\times\Omega)$ in the norm
$$
\|\mathsf u\|_{
H^k_{\beta}((-\infty,{R})\times\Omega)}=\left(\sum_{\ell\leq
k}\int_{-\infty}^{R} \|e^{-\beta x}\partial_x^\ell\mathsf
u(x)\|^2_{H^{k-\ell}(\Omega)}\,dx\right)^{1/2}.
$$
We also set
$L^2_\beta\bigl((-\infty,{R})\times\Omega\bigr)=H^0_{\beta}((-\infty,{R})\times\Omega)$
and denote by $H^{1/2}_\beta((-\infty,{R})\times\partial\Omega)$ the
space of traces on $(-\infty,{R})\times\partial\Omega$ of the
functions in $H^1_\beta((-\infty,{R})\times\Omega)$.
\begin{lemma}\label{al}
Assume that $\lambda\in\mathcal O_\alpha\setminus\Bbb R$,
$\mu_0\in\Bbb R\setminus\sigma(\Delta^N_\Omega)$, and $\beta$ is in
the interval~\eqref{interval}. Then for any $\mathsf f\in
L^2_\beta\bigl((-\infty,{R})\times\Omega\bigr)$ and $\mathsf g\in
H_\beta^{1/2}((-\infty,{R})\times\partial\Omega)$ there exists a
 solution $\mathsf u$ to
 the second limit problem
\begin{equation}\label{eq}
(\Delta_\Omega-(1+\lambda)^{-2}\partial^2_x -\mu_0)\mathsf u=\mathsf
f\text{ on } (-\infty,{R})\times\Omega,\quad
\partial_\eta \mathsf u=\mathsf g \text{ on }
(-\infty,{R})\times\partial\Omega,\quad
 \mathsf u=0\text{ on }
\{{R}\}\times\Omega,
\end{equation}
satisfying the estimate
\begin{equation}\label{eqest} \bigl\|\mathsf
u\|_{ H^2_{\beta}((-\infty,{R})\times\Omega)}\leq C\bigl(\|\mathsf
f\|_{ L^2_\beta((-\infty,{R})\times\Omega)}+\|\mathsf g\|_{
H_\beta^{1/2}((-\infty,{R})\times\partial\Omega)}\bigr),
\end{equation}
 where the constant $C$ is independent of $\mathsf f$ and $\mathsf
g$. Here $\partial_\eta$ is the operator of the Neumann boundary
conditions on $(-\infty,{R})\times\partial\Omega$ induced by the
product metric $dx\otimes dx+\mathfrak h$ on
$(-\infty,{R})\times\Omega$.
\end{lemma}
\begin{proof} Without loss of generality we can assume that ${R}=0$ (the
general case can be obtained by the change of variables $x\mapsto
x-{R}$). We give only a sketch of the proof as it is essentially
based on a well known argument, details can be found e.g.
in~\cite{KozlovMazyaRossmann}.

  Assume that $(\mathsf f,\mathsf g)\in
C_c^\infty(\Bbb R_-\times\Omega)\times C_c^\infty(\Bbb
R_-\times\partial \Omega)$ and extend it to a function in
$C_c^\infty(\Bbb R\times\Omega)\times C_c^\infty(\Bbb
R\times\partial \Omega)$ by setting $\bigl(\mathsf f,\mathsf
g\bigr)(-x) =-\bigl(\mathsf f,\mathsf g\bigr)(x)$ for $x<0$.
 The Fourier transform
$\mathcal F_{x\mapsto \xi}(\mathsf f,\mathsf g)$ is an entire
function of $\xi$ with values in $L^2(\Omega)\times
H^{1/2}(\partial\Omega)$ decaying faster than $|\xi|^{-k}$ with any
$k>0$ as $\xi$ tends at infinity in any strip $|\Im\xi|<C$. Since
$\mu_0-(1+\lambda)^{-2}\xi^2$ with $0\leq\Im\xi<\beta$ is not an
eigenvalue of $\Delta^N_\Omega$ and the strongly elliptic operator
$\Delta_\Omega-(1+\lambda)^{-2}\partial^2_x$ on $\Bbb R\times\Omega$
with the operator $\partial_\eta$ on $\Bbb R\times\partial\Omega$
set up a regular elliptic problem, the elliptic coercive estimate
\begin{equation}\label{estF}
\sum_{p=0}^2|\xi|^{2p}  \|\Psi\|_{ H^{2-p}(\Omega)}^2 \leq  C
\bigl(\|((1+\lambda)^{-2}\xi^2+\Delta_\Omega -\mu) \Psi\|^2_{
L^2(\Omega)} +|\xi|\|\partial_\eta \Psi\|_{
L^2(\partial\Omega)}^2+\|\partial_\eta \Psi\|_{
H^{1/2}(\partial\Omega)}^2\bigr)
\end{equation}
with parameter $0\leq\Im\xi<\beta$ is valid. Moreover, the operator
$$
\bigl(\Delta_\Omega+(1+ \lambda)^{-2}\xi^2 -\mu_0, \partial_\eta
\bigr): H^2(\Omega)\to H^0(\Omega)\times H^{1/2}(\partial\Omega),
\quad 0\leq \Im\xi<\beta,
$$
yields an isomorphism and
\begin{equation}\label{fourier}
\mathsf u(x)=\mathcal F^{-1}_{\xi\mapsto x} \bigl(\Delta_\Omega+(1+
\lambda)^{-2}\xi^2 -\mu_0,
\partial_\eta \bigr)^{-1} \mathcal F_{x\mapsto \xi}(\mathsf f,\mathsf
g)
\end{equation}
is a unique in $L^2(\Bbb R\times\Omega)$ solution to the Neumann
problem
$$
(\Delta_\Omega-(1+\lambda)^{-2}\partial_x^2-\mu_0)\mathsf u=\mathsf
f \text{ on } \Bbb R\times\Omega, \quad
\partial_\eta \mathsf u=\mathsf g\text{ on } \Bbb
R\times\partial\Omega.
$$
Usual argument on smoothness of solutions to elliptic problems gives
$\mathsf u\in C^\infty(\Bbb R\times\Omega)$. Since $x\mapsto
\bigl(\mathsf f,\mathsf g\bigr)(x)$ is odd it follows that $\mathcal
F_{x\mapsto \xi}(\mathsf f,\mathsf g)$ is an odd function of $\xi$.
Therefore
 $x\mapsto\mathsf u(x)$ is odd and
$\mathsf u(0)=0$.

The Cauchy integral theorem and the estimate~\eqref{estF} allow to
replace the contour of integration $\Bbb R$ in~\eqref{fourier} by
$\{\xi\in\Bbb C:\Im\xi=\beta\}$. Then
 the Parseval equality implies that $\mathsf u$ is a
solution to the problem~\eqref{eq} satisfying the
estimate~\eqref{eqest}. By continuity our construction extends to
all $\mathsf f\in L^2_\beta\bigl(\Bbb R_-\times\Omega\bigr)$ and
$\mathsf g\in H_\beta^{1/2}(\Bbb R_-\times\partial\Omega)$.
\end{proof}
Now we are in position to prove Theorem~\ref{thMR}.
\begin{proof} The proof is carried out using the compound expansion method. We say that $w\in
H^2(\mathcal M_{R})$ is an approximate solution to the
non-homogeneous problem
\begin{equation}\label{MR1}
({^\lambda\!\Delta}-\mu_0)v=F \text{ on } \mathcal
M_{R},\quad{^\lambda\!\mathcal N} v=G \text{ on } \partial \mathcal
M_{R}^-,\quad v=0 \text{ on } \{{R}\}\times\Omega
\end{equation}
 if the following conditions are satisfied:
\begin{itemize}
\item[i.] The estimate
$$
\|w\|_{ H^2(\mathcal M_{R})} \leq c\|(F,G)\|_{L^2(\mathcal
M_{R})\times H^{1/2}(\partial\mathcal M_{R}^-)}
$$
holds with an independent of $F$, $G$, and ${R}$ constant $c$;
\item[ii.] The estimate
\begin{equation}
\|({^\lambda\!\Delta}-\mu_0,{^\lambda\!\mathcal N})w-(F,G)\|_{
L^2(\mathcal M_{R})\times H^{1/2}(\partial\mathcal M_{R}^-)} \leq
C_{R}\|(F,G)\|_{L^2(\mathcal M_{R})\times H^{1/2}(\partial\mathcal
M_{R}^-)}
\end{equation}
is valid, where the constant $C_{R}$ is independent of $F$ and $G$
and $C_{R}\to 0$ as ${R}\to+\infty$.
\end{itemize}
Due to Condition i $w_{R}$ continuously depends on $F$ and $G$.
Condition ii implies that the discrepancy, left by $w_{R}$ in the
problem~\eqref{MR1}, tends to zero as ${R}\to+\infty$. Once an
approximate solution is found, it is not hard to verify the
assertion of the theorem.

Let $\rho\in C^\infty (\Bbb R)$ be a cutoff function such that
$\rho(x)=1$ for $x\leq 0$ and $\chi(x)=0$ for $x\geq 1/2$. We set
$\rho_{R}=\rho(x-{R})$ on $\Pi$ and $\varrho_{R}=1$ on $\mathcal
M\setminus\Pi$. Let $(f,g)=\varrho_{{R}/2} (F,G)$ and $(\mathsf
f,\mathsf g)=(1-\varrho_{R})(F,G)$. Extend $(F,G)$ from $\mathcal
M_{R}\times\partial\mathcal M_{R}^-$ to $\mathcal
M\times\partial\mathcal M$ and $(\mathsf f,\mathsf g)$ from
$(0,{R})\times\Omega\times(0,{R})\times\partial\Omega$ to
$(-\infty,{R})\times\Omega\times(-\infty,{R})\times\partial\Omega$
by zero. We already know that
$\mu_0\notin\sigma({^\lambda\Delta_\beta})$ and hence~\eqref{ad}
yields an isomorphism for $\beta$ in the interval
in~\eqref{interval} (see the proof of Theorem~\ref{T1}.2). We find
an approximate solution $w$ compounded of
$u_{\lambda}=e^{-\beta\mathsf
s}({^\lambda\!\Delta_\beta}-\mu_0,{^\lambda\!\mathcal
N_\beta})^{-1}e^{\beta\mathsf s}(f,g)$ and a solution $\mathsf u\in
H^2_\beta((-\infty,{R})\times\Omega)$ to the equation~\eqref{eq} in
the form
$$
w=\varrho_{R} u_{\lambda}+(1-\varrho_{{R}/3})\mathsf u;
$$
here the second term in the right hand side is extended from
$(-\infty,{R})\times\Omega$ to $\mathcal M_{R}$ by zero. On the
support of $(f,g)$ we have $e^{\beta\mathsf s}\leq C e^{\beta
{R}/2}$ and on the support of $(\mathsf f,\mathsf g)$ we have
$e^{-\beta x}\leq C e^{-\beta {R}/2}$ uniformly in ${R}$. Hence
\begin{equation}\label{rhs}
\begin{aligned}
\|e^{\beta\mathsf s}(f,g)\|_{ L^2(\mathcal M)\times H^{1/2}(\partial
M)}+e^{\beta {R}}\bigl\|(\mathsf f,\mathsf g)\|_{
H^0_\beta((-\infty,{R})\times\Omega)\times
H^{1/2}_\beta((-\infty,{R})\times\partial\Omega)}
\\
\leq C e^{\beta {R}/2}\| (F,G)\|_{ L^2(\mathcal M_{R})\times
H^{1/2}(\partial\mathcal M_{R}^-)}
\end{aligned}
\end{equation}
with an independent of ${R}$ and $(F,G)$ constant $C$. Thanks
to~\eqref{eqest} and~\eqref{rhs} we can conclude that
$$
\begin{aligned}
\|w\|_{ H^2(\mathcal M_{R})}^2\leq \|\varrho_{R}
u_{\lambda}\|_{H^2(\mathcal M)}^2+\|(1-\rho_{{R}/3})\mathsf
u\|_{H^2((-\infty,{R})\times\Omega)}^2 \leq C\| (F,G)\|^2_{
L^2(\mathcal M_{R})\times H^{1/2}(\partial\mathcal M_{R}^-)}
\end{aligned}
$$
and Condition i is satisfied.

Let us verify Condition ii. We have
\begin{equation}\label{9.5}
\begin{aligned}
({^\lambda\!\Delta}-\mu_0,{^\lambda\!\mathcal
N})w-(F,G)=[({^\lambda\!\Delta},{^\lambda\!\mathcal
N}),\varrho_{R}]u_\lambda+\bigl((1+\lambda)^{-2}[\partial^2_x,\rho_{{R}/3}]\mathsf
u,0\bigr)
\\
+\bigl({^\lambda\!\Delta}-\Delta_\Omega+(1+\lambda)^{-2}\partial^2_x,{^\lambda\!\mathcal
N}-\partial_\eta\bigr)(1-\rho_{{R}/3})\mathsf u.
\end{aligned}
\end{equation}
On the support of the commutator
$[({^\lambda\!\Delta},{^\lambda\!\mathcal N}),\varrho_{R}]u_\lambda$
the weight $e^{\beta\mathsf s}$ is bounded from below by $c e^{\beta
{R}}$ uniformly in ${R}>0$. As a consequence,
\begin{equation}\label{9.6}
\bigl\|[({^\lambda\!\Delta},{^\lambda\!\mathcal
N}),\varrho_{R}]u_\lambda\bigr\|_{ L^2(\mathcal M_{R})\times
H^{1/2}(\partial\mathcal M_{R}^-)}\leq C_1 e^{-\beta
{R}}\|e^{\beta\mathsf s}u_{\lambda}\|_{H^2(\mathcal M)}\leq C_2
e^{-\beta {R}}\|e^{\beta\mathsf s}(f,g)\|_{ L^2(\mathcal M)\times
H^{1/2}(\partial\mathcal M)}.
\end{equation}
 Now we estimate the second term in
the right hand side of~\eqref{9.5}. On the support of
$[\partial^2_x,\rho_{{R}/3}]\mathsf u$ we have $e^{-\beta x}\geq C
e^{-\beta {R}/3}$. Thanks to~\eqref{eqest} we obtain
\begin{equation}\label{9.7}
\begin{aligned}
\bigl\|(1+\lambda)^{-2}[\partial^2_x,\rho_{{R}/3}]\mathsf u\bigr\|_{
L^2(\mathcal M_{R})}\leq C_1 e^{\beta {R}/3} \bigl\|
u;H^2_\beta\bigl((-\infty,{R})\times\Omega\bigr)\bigr\| \\  \leq
C_2e^{\beta {R}/3}\bigl\|(\mathsf f,\mathsf g)\|_{
H^0_\beta((-\infty,{R})\times\Omega)\times
H^{1/2}_\beta((-\infty,{R})\times\Omega)}.
\end{aligned}
\end{equation}
Finally, consider the last term in the right hand side
of~\eqref{9.5}. On the support of $(1-\rho_{{R}/3})\mathsf u$ the
coefficients of the operator
$\bigl({^\lambda\!\Delta}-\Delta_\Omega+(1+\lambda)^{-2}\partial^2_x,{^\lambda\!\mathcal
N}-\partial_\eta\bigr)$ uniformly tend to zero as ${R}\to +\infty$;
see~\eqref{STAB-} and~\eqref{oper},~\eqref{bc}. This together
with~\eqref{eqest} gives
\begin{equation}\label{9.8}
\begin{aligned}
\bigl\|
\bigl({^\lambda\!\Delta}-\Delta_\Omega+(1+\lambda)^{-2}\partial^2_x,{^\lambda\!\mathcal
N}-\partial_\eta\bigr)(1-\rho_{{R}/3})\mathsf u\bigr\|_{L^2(\mathcal
M_{R})\times H^{1/2}(\partial\mathcal M_{R}^-)} \\
\leq c_{R}\bigl\|(\mathsf f,\mathsf
g)\|_{H^2_0((-\infty,{R})\times\Omega)\times
H^{1/2}_0((-\infty,{R})\times\partial\Omega)},
\end{aligned}
\end{equation}
where $c_{R}\to 0$ as ${R}\to+\infty$. From~\eqref{rhs}--\eqref{9.8}
it follows that $w$ meets Condition ii. Thus $w$ is  an approximate
solution to the problem~\eqref{MR1}.

Now we are in position to prove the assertion of the theorem.
Observe that $({^\lambda\!\Delta},{^\lambda\!\mathcal
N})w-(f,g)=\mathfrak O_{R} (f,g)$ with some operator $\mathfrak
O_{R}$ in $L^2(\mathcal M_{R})\times H^{1/2}(\partial \mathcal
M_{R}^-)$, whose norm $\|\mathfrak O_{R}\|$ tends to zero as
${R}\to+\infty$ because of Condition ii on $w$. For all ${R}>{R}_0$
we have $\|\mathfrak O_{R}\|\leq\epsilon<1$. Hence there exists the
inverse $(I+\mathfrak O_{R})^{-1}$ and its norm is bounded by the
constant $1/(1-\epsilon)$ uniformly in ${R}>{R}_0$. We set $(\tilde
f,\tilde g)=(I+\mathfrak O_{R})^{-1}(\tilde f,\tilde g)$. In the
same way as before we construct the approximate solution $w$ for the
problem~\eqref{MR1}, where $(f,g)$ is replaced by $(\tilde f,\tilde
g)$. Then for $v=w$ we have
$({^\lambda\!\Delta}-\mu_0,{^\lambda\!\mathcal N})v=(\tilde f,\tilde
g)+\mathfrak O_{R} (\tilde f,\tilde g)=(f,g)$ and
$$
 \|v;H^2(\mathcal
M_{R})\|\leq c\|(\tilde f,\tilde g)\|_{L^2(\mathcal M_{R})\times
H^{1/2}(\partial\mathcal M_{R}^-)}\leq
c/(1-\epsilon)\|(f,g)\|_{L^2(\mathcal M_{R})\times
H^{1/2}(\partial\mathcal M_{R}^-)},$$ where $C$ is independent of
${R}>{R}_0$. In particular for ${R}>{R}_0$ and $f\in L^2(\mathcal
M_{R})$ there exists a solution $v\in H^2(\mathcal M_{R})$ to the
problem~\eqref{MR} satisfying~\eqref{estMR}. This solution is unique
as a similar argument shows that the adjoint problem is solvable in
$H^2(\mathcal M_{R})$ for any right hand side in $L^2(\mathcal
M_{R})$.
\end{proof}

In contrast to infinite PMLs, finite PMLs are not non-reflective.
Reflections produce a non-zero difference in $\mathcal M_{r}$
between solutions $u_\pm$ satisfying the limiting absorption
principle~\eqref{LAP} and solutions $v\in H^2(\mathcal M_{R})$ to
the problem~\eqref{MR} with $F=f_\lambda$ and $\Im\lambda\lessgtr
0$. In the next theorem we prove that this difference (error) decays
with an exponential rate as ${R}\to+\infty$. In other words, we
prove exponential convergence of the PML method. The problem with
finite PMLs can be solved numerically with the help of finite
element solvers; discretization produces yet another error that we
do not estimate here.
\begin{theorem}\label{convergence}  Let  $\lambda\in\mathcal
O_\alpha\setminus\Bbb R$ and let $\beta$ be in the interval
in~\eqref{interval}. Assume that $\mu_0\in\Bbb
R\setminus\sigma(\Delta^N_\Omega)$ is not an eigenvalue of the
selfadjoint Neumann Laplacian $^0\!\Delta$ in $L^2(\mathcal M)$.
Then there exists ${R}_0>0$ such that  for any $f\in\mathcal A$ and
for all ${R}>{R}_0$ a unique solution $v_{R}\in H^2_0(\mathcal
M_{R})$ of the problem~\eqref{MR} with $F=f_\lambda$
 converges on  $\mathcal
M_{r}$
to the outgoing solution $u_-$ (resp. the incoming solution $u_+$)
of the equation $({^0\!\Delta}-\mu_0)u=f$ in the case $\Im\lambda>0$
(resp. $\Im\lambda<0$)
in the  sense that as ${R}\to+\infty$ the estimate
\begin{equation}\label{fin}
\|u_\pm-v_{R}\|_{ H^2(\mathcal M_{r})}^2\leq Ce^{-2\beta
{R}}\|e^{\beta\mathsf s}f_\lambda\|^2
\end{equation}
holds with a constant $C$ independent of ${R}>{R}_0$ and $f$.
\end{theorem}

Let us recall that the set $\mathcal A$ of analytic functions is
dense in $L^2(\mathcal M)$. In particular, our construction and
Theorem~\ref{convergence} remain valid for any $f\in L^2(\mathcal
M)$ supported in $\mathcal M_{r}$.

\begin{proof} By Theorem~\ref{T1} it suffices
to prove the estimate~\eqref{fin} with $u_\pm$ replaced by $
u_{\lambda}$. Let $\rho\in C^\infty (\Bbb R)$ be a cutoff function
such that $\rho(x)=1$ for $x\leq 0$ and $\chi(x)=0$ for $x\geq 1/2$.
We set $\rho_{R}(x)=\rho(x-{R})$ on $\Pi$ and $\rho_{R}=1$ on
$\mathcal M\setminus\Pi$.  Then $\rho_{R} u_{\lambda}-v_{R}=
u_{\lambda}-v_{R}$ on $\mathcal M_R$ and the difference $\rho_{R}
u_{\lambda}-v_{R}\in H^2(\mathcal M_{R})$ satisfies~\eqref{MR} with
$F=(\rho_{R}-1)f_{\lambda}+[{^\lambda\!\Delta},\rho_{R}]u_{\lambda}$.
Observe that
$$
\|(\varrho_{R}-1)f_{\lambda}\|_{ L^2(\mathcal M_{R})}\leq C
e^{-\beta {R}}\|e^{\beta\mathsf s}f_{\lambda}\|, \quad
 \|[{^\lambda\!\Delta},\rho_{R}]u_{\lambda}\|_{L^2(\mathcal
M_{R})}\leq C e^{-\beta {R}}\|e^{\beta \mathsf s}
u_{\lambda}\|_{H^2(\mathcal M)}.
$$
This together with~\eqref{interval} gives
\begin{equation}\label{fin-2}
\|F\|_ {L^2(\mathcal M_{R})}\leq c e^{-\beta {R}}\|e^{\beta\mathsf
s}f_{\lambda}\|.
\end{equation}
By Theorem~\ref{thMR} we have
$$
\|\rho_{R} u_{\lambda}-v_{R}\|_{ H^2(\mathcal M_{R})}\leq C
\|F\|_{L^2(\mathcal M_{R})},\quad {R}>{R}_0.
$$
This together with
$$
\|u_\pm-v_{R}\|_{L^2(\mathcal M_{r})}\leq \|\rho_{R}
u_{\lambda}-v_{R}\|_{H^2(\mathcal M_{R})},\quad {R}>{R}_0>{r},
$$
 and~\eqref{fin-2} completes the proof of~\eqref{fin}.
\end{proof}

\section*{Funding} This work was supported by  the Academy of Finland  [108898].

\end{document}